\documentclass[12pt,twoside,reqno]{amsart}

\usepackage{bm}

\usepackage{amssymb,amsmath,amstext,amsthm,amsfonts,amscd,xcolor}

\usepackage{dsfont}

\usepackage{amsthm, enumerate}

\usepackage[ansinew]{inputenc}
\usepackage{graphicx}
\usepackage[mathscr]{eucal}
\usepackage{hyperref}

\newcommand{\Supp}{{\rm \Sigma}}
\newcommand{\R}{\mathbb{R}}
\newcommand{\Su}{\mathbb{S}}

\newcommand{\N}{\mathbb{N}}
\newcommand{\Z}{\mathbb{Z}}
\newcommand{\Q}{\mathbb{Q}}
\newcommand{\T}{\mathbb{T}}
\newcommand{\SL}{{\rm SL}}
\newcommand{\GL}{{\rm GL}}

\newcommand{\tr}{\mbox{tr}}

\newcommand{\Pp}{\mathbb{P}}

\newcommand{\EE}{\mathbb{E}}

\newcommand{\Sigmah}{{\Omega}}

\theoremstyle{plain}
\newtheorem{theorem}{Theorem}[section]
\newtheorem{proposition}{Proposition}[section]
\newtheorem{corollary}[proposition]{Corollary}
\newtheorem{lemma}[proposition]{Lemma}
\theoremstyle{definition}

\newtheorem{definition}{Definition}[section]
\newtheorem*{theorem*}{Theorem}

\theoremstyle{definition}
\newtheorem{remark}{Remark}[section]
\newtheorem{example}[theorem]{Example}
\numberwithin{equation}{section}

\newcommand{\abs}[1]{\left| #1 \right|} % absolute value
\newcommand{\norm}[1]{\left\|#1\right\|} % norm
\newcommand{\normtwo}[1]{% Peter Grill norm @tex.stackexchange.com
{\left\vert\kern-0.25ex\left\vert\kern-0.25ex\left\vert #1
    \right\vert\kern-0.25ex\right\vert\kern-0.25ex\right\vert} }

\newcommand{\medp}{\hat{\mathfrak{v}}}

\newcommand{\avg}[1]{\left< #1 \right>} % average
\newcommand{\om}{\omega}

\makeatletter
\newsavebox\myboxA
\newsavebox\myboxB
\newlength\mylenA

\newcommand*\xoverline[2][0.75]{%
    \sbox{\myboxA}{$\m@th#2$}%
    \setbox\myboxB\null% Phantom box
    \ht\myboxB=\ht\myboxA%
    \dp\myboxB=\dp\myboxA%
    \wd\myboxB=#1\wd\myboxA% Scale phantom
    \sbox\myboxB{$\m@th\overline{\copy\myboxB}$}%  Overlined phantom
    \setlength\mylenA{\the\wd\myboxA}%   calc width diff
    \addtolength\mylenA{-\the\wd\myboxB}%
    \ifdim\wd\myboxB<\wd\myboxA%
       \rlap{\hskip 0.5\mylenA\usebox\myboxB}{\usebox\myboxA}%
    \else
        \hskip -0.5\mylenA\rlap{\usebox\myboxA}{\hskip 0.5\mylenA\usebox\myboxB}%
    \fi}
\makeatother

\newcommand{\comp}{^{\complement}}

\newcommand{\Proj}{\mathbb{P}(\R^m)}

\newcommand{\FF}{\mathscr{F}}
\newcommand{\SO}{{\rm SO}}

\newcommand{\ind}{\mathds{1}}

\newcommand{\B}{\mathscr{B}}

\newcommand{\dist}{{\rm dist}}

\newcommand{\Bscr}{\mathscr{B}}
\newcommand{\Nscr}{\mathscr{N}}
\newcommand{\QP}{\mathrm{Q}}

\newcommand\restr[2]{{% we make the whole thing an ordinary symbol
  \left.\kern-\nulldelimiterspace % automatically resize the bar with \right
  #1 % the function
  \vphantom{\big|} % pretend it's a little taller at normal size
  \right|_{#2} % this is the delimiter
  }}
\newcommand{\Hscr}{\mathscr{H}}

\newcommand{\Prob}{\mathrm{Prob}}
\newcommand{\Ascr}{\mathscr{A}}
\newcommand{\Gscr}{\mathscr{G}}
\newcommand{\Lscr}{\mathscr{L}}
\newcommand{\supp}{\mathrm{supp}}
\newcommand{\Cscr}{C^0}
\newcommand{\Qop}{\mathcal{Q}}
\newcommand{\diam}{\mathrm{diam}}
\newcommand{\medr}{\mathfrak{v}}
\newcommand{\gpr}{\sigma}

\newcommand{\Mscr}{\mathscr{M}}
\newcommand{\alfa}{\mathfrak{a}}
\newcommand{\muh}{{\nu}}

\newcommand{\etah}{{\tilde \eta}}
\newcommand{\PF}{\hat F}
\newcommand{\PFh}{\hat F}
\newcommand{\Pmeas}{\mathrm{m}}
\newcommand{\Qmeas}{\mathfrak{m}}
\newcommand{\Xplus}{{X^+}}
\newcommand{\fle}{L_1}
\newcommand{\lexp}{L}
\newcommand{\dle}{\hat L}

\newcommand{\Kscr}{\mathscr{K}}
\newcommand{\Sscr}{\mathscr{S}}
\newcommand{\Vscr}{\mathscr{V}}

\newcommand{\Bcal}{\mathcal{B}}

\title[Generic Mixed Random-Quasiperiodic Cocycles]{Furstenberg Theory of Mixed Random-Quasiperiodic Cocycles}

\date{}

\begin{document}

\author[A. Cai]{Ao Cai}
\address{Departamento de Matem\'atica and CMAFcIO\\
Faculdade de Ci\^encias\\
Universidade de Lisboa\\
Portugal
}
\email{acai@fc.ul.pt}

\author[P. Duarte]{Pedro Duarte}
\address{Departamento de Matem\'atica and CMAFcIO\\
Faculdade de Ci\^encias\\
Universidade de Lisboa\\
Portugal
}
\email{pmduarte@fc.ul.pt}

\author[S. Klein]{Silvius Klein}
\address{Departamento de Matem\'atica, Pontif\'icia Universidade Cat\'olica do Rio de Janeiro (PUC-Rio), Brazil}
\email{silviusk@puc-rio.br}

\begin{abstract}
We derive a criterion for the positivity of the maximal Lyapunov exponent of generic mixed random-quasiperiodic linear cocycles, a model introduced in a previous work. This result is applicable to cocycles corresponding to Schr\"odinger operators with randomly perturbed quasiperiodic potentials. Moreover, we establish an average uniform convergence to the Lyapunov exponent in the Oseledets theorem. 
\end{abstract}

\maketitle

%\tableofcontents

\section{Introduction and main statements}\label{intro}
In a previous work~\cite{CDK-paper1} we introduced the concept of mixed random-quasiperiodic linear cocycles and studied some of its basic properties.
In simple terms, we can think of the iterates of such dynamical systems as {\em random products} (compositions) {\em of quasiperiodic cocycles}. 

One of the goals of this paper is to extend the results of what we refer to as Furstenberg's theory on random products of matrices to this more general setting. That is, for our mixed systems, we derive a Furstenberg-type formula to represent the maximal Lyapunov exponent; we study Kifer's non-random filtration for generic cocycles; we obtain a criterion for the positivity of the maximal Lyapunov exponent and we establish its continuity under some generic assumptions.

Such topics have previously been studied in other settings, more general or related to ours, see the monograph of Y. Kifer~\cite{Kifer-book} and the survey of A. Furman~\cite{FurmanSurvey}. More recently, the positivity and the continuity of the Lyapunov exponent for a more particular mixed model was studied by J. Bezerra and M. Poletti~\cite{Jamerson-M}. Furthermore, during a recent workshop\footnote{New trends in Lyapunov exponents, held in Lisbon in February of 2022.} we became aware of some intersection between results in this paper and the work in preparation by A. Gorodetski and V. Kleptsyn on non-stationary versions of Furstenberg's theorem on random matrix products (see page 18 of the final report~\cite{GK-birs} of a BIRS workshop where their results were announced). 

As such, not all of the results on Furstenberg's theory presented in this paper are completely new, but we mean to provide a coherent, relatively self contained exposition of these concepts and results in the specific context of mixed random-quasiperiodic cocycles.

Moreover, we also establish the uniform convergence to the Lyapunov exponent of the expected value (in the random component) of the geometric averages of the iterates of the cocycle. Together with the aforementioned results of Furstenberg's theory, this forms the basis for the study of the statistical properties of the iterates of such mixed cocycles, which will be considered in a future work.

\medskip

A relevant example of mixed random-quasiperiodic cocycle arises from considering certain discrete Schr\"odinger operators, which represent an important motivation for this work. More precisely, given a potential function $v \in C^0 (\T^d, \R)$, where $\T^d = (\R/\Z)^d$ is the $d$-dimensional torus, an ergodic frequency $\alpha \in \T^d$, an i.i.d. sequence of real random variables $\{w_n\}_{n\in\Z}$ and a phase $\theta \in \T^d$,  consider the discrete Schr\"odinger operator $H (\theta)$ on $l^2 (\Z)$, defined by
\begin{equation}\label{op1}
(H (\theta)  \psi)_n = - \psi_{n+1} - \psi_{n-1} + \left( v (\theta + n \alpha) + w_n \right) \, \psi_n \quad \forall n \in \Z ,
\end{equation}
for all $\psi = \{\psi_n\}_{n\in\Z} \in l^2 (\Z)$.

The i.i.d. sequence $\{w_n\}$ can be interpreted as a random perturbation of the quasiperiodic potential $\{v_n (\theta)\}$, where $v_n (\theta) := v (\theta + n \alpha)$.

\medskip

We may instead randomize the frequency, that is, given an i.i.d sequence $\{\alpha_n\}_{n\in\Z}$ of random variables with values on the torus $\T^d$, 
consider the discrete Schr\"odinger operator
\begin{equation}\label{op2}
(H (\theta)  \psi)_n = - \psi_{n+1} - \psi_{n-1} + v (\theta + \alpha_0 + \cdots + \alpha_{n-1} )  \, \psi_n \quad \forall n \in \Z .
\end{equation}

We may, of course, both randomize the frequency and randomly perturb the resulting potential, thus obtaining the mixed Schr\"odinger operator
\begin{equation}\label{op3}
(H (\theta)  \psi)_n = - \psi_{n+1} - \psi_{n-1} + \left( v_n (\theta) + w_n \right) \, \psi_n \quad \forall n \in \Z,
\end{equation}
where
$$v_n (\theta) = v (\theta + \alpha_0 + \cdots + \alpha_{n-1} ) \, .$$

Given any of the above models, consider the corresponding Schr\"odinger (or eigenvalue) equation
$$H (\theta) \, \psi = E \, \psi \, .$$
When solved formally, this equation gives rise to the mixed random-quasiperiodic multiplicative process
$$\Ascr_E^n (\theta) := \prod_{j=n-1}^0 \begin{bmatrix} v_{j} (\theta)  + w_j - E & -1 \\ 1 & 0 \end{bmatrix} \, ,$$
whose average asymptotic behavior can be studied in the general framework of this paper.

In particular, under very mild assumptions on the underlying probability distributions, we establish the positivity for all energies $E$ of the Lyapunov exponent of the Schr\"odinger operators~\eqref{op1},~\eqref{op2} and~\eqref{op3}, as well as its continuity as a function of the energy (and of the potential function).

 The positivity of the Lyapunov exponent for these mixed models shows that in some sense, the randomness is dominant over the deterministic (quasiperiodic) part of the potential, that is, mixed random-quasiperiodic Schr\"odinger operators / linear cocycles tend to behave more like their random counterparts (the Lyapunov exponent of a quasiperiodic Schr\"odinger operator is zero in certain regimes, while that of random Schr\"odinger operators is always positive).
 
 Furthermore, by Kotani's theory, the positivity of the Lyapunov exponent implies the absence, almost surely, of absolutely continuous spectrum for the operators~\eqref{op1},~\eqref{op2} and~\eqref{op3}, and represents a strong indication of their localization, a question to be considered in a future project.
 
 \medskip
 
We now formally introduce our model.  For more details see~\cite{CDK-paper1}.
 
\subsection*{The base dynamics}  Let $(\Sigmah, \Bcal)$ be a standard Borel space
and let $\muh \in \Prob_c (\Sigmah)$ be a compactly supported Borel probability measure on $\Sigmah$. Regarding $(\Sigmah, \muh)$ as a space of symbols, we consider the corresponding (invertible) Bernoulli system $\left(X, \sigma, \muh^\Z \right)$, where $X := \Sigmah^\Z$ and $\sigma \colon X \to X$ is the (invertible) Bernoulli shift: for $\om = \{ \om_n \}_{n\in\Z}  \in X$,  
$\sigma \om  :=  \{ \om_{n+1} \}_{n\in\Z}$. Consider also its non invertible factor on $X^+ := \Sigmah^\N$.

Let $\T^d = \left(\R/\Z \right)^d$ be the torus of dimension $d$, and denote by $m$ the Haar measure on its Borel $\sigma$-algebra.
 
Given a continuous function $\alfa \colon \Sigmah \to \T^d$, the skew-product map
\begin{equation}\label{base map}
f \colon X \times \T^d \to X  \times \T^d \, , \quad f (\om, \theta) := \left( \sigma \om, \theta + \alfa ( \om_0) \right)
\end{equation}
will be referred to as a mixed random-quasiperiodic (base) dynamics. 

This map preserves the measure $\muh^\Z\times m$ and it is the natural extension of the non invertible map on
$X^+ \times \T^d$ which preserves the measure $\muh^\N\times m$ and is defined by the same expression.

We will call the measure $\muh$ ergodic, or ergodic with respect to $f$ when the mixed random-quasiperiodic system $\left( X \times \T^d, f, \muh^\Z\times m \right)$ is ergodic. See~\cite[Section 2]{CDK-paper1} for various characterizations of the ergodicity of this system.

Note that in the particular case when $\alfa$ is the constant function $\alfa (\om_0) \equiv \alpha$, the corresponding mixed system $f$ is simply the product between the Bernoulli shift and the torus translation by $\alpha$, whose ergodicity implies that of $f$.

Another relevant example is obtained by choosing a probability space $(S, \rho)$ and a measure $\mu \in \Prob (\T^d)$ and letting $\Sigmah := \T^d \times S$, $\muh := \mu \times \rho$ and $\alfa$ be the projection in the first coordinate. Using Proposition 2.1 and Theorem 2.3 in~\cite{CDK-paper1}, the ergodicity of $\muh$ is equivalent to the existence, for every $k \in \Z^d \setminus \{0\}$, of a frequency $\alpha \in \supp (\mu)$ such that $\avg{k, \alpha} \notin \Z$.

\subsection*{The group of quasiperiodic cocycles} Given a  frequency $\alpha \in \T^d$, let $\tau_\alpha (\theta) = \theta + \alpha$ be the corresponding ergodic translation on $\T^d$. Consider $A \in C^0 (\T^d, \SL_m (\R))$ a continuous matrix valued function on the torus. A quasiperiodic cocycle is a skew-product map of the form 
$$\T^d \times \R^m \ni (\theta, v) \mapsto \left( \tau_\alpha (\theta),  A(\theta) v \right) \in \T^d \times \R^m \, .$$

This cocycle can thus be identified with the pair $(\alpha, A)$. Consider the set
$$\Gscr=\Gscr(d,m):= \T^d\times C^0(\T^d,\SL_m(\R))$$
of all quasiperiodic cocycles.  

This set is a Polish metric space when equipped with the product metric (in the second component we consider the uniform distance). The space $\Gscr$ is also a group, and in fact a topological group, with the natural composition and inversion operations
\begin{align*}
(\alpha,A)\circ (\beta, B) &:= (\alpha+\beta, (A\circ \tau_\beta) \, B)  \\
(\alpha,A)^{-1} &:= ( -\alpha, (A\circ \tau_{-\alpha})^{-1} ) \, .
\end{align*}

Given $\muh \in \Prob_c (\Gscr)$ let $\om = \left\{ \om_n \right\}_{n\in\Z}$, $\om_n = (\alpha_n, A_n)$ be an i.i.d. sequence of random variables  in $\Gscr$  with law $\muh$. Consider the corresponding multiplicative process in the group $\Gscr$
\begin{align*}
\Pi_n & = \om_{n-1} \circ \cdots \circ \om_1 \circ \om_0 \\
& = \left( \alpha_{n-1} + \cdots + \alpha_1 + \alpha_0, \, ( A_{n-1} \circ \tau_{\alpha_{n-2} + \cdots + \alpha_0}) \cdots (A_1 \circ \tau_{\alpha_0} ) \, A_0   \right) \, .
\end{align*}

In order to study this process in the framework of ergodic theory, we model it by the iterates of a  linear cocycle.

 \subsection*{The fiber dynamics} Given $\muh \in \Prob_c (\Gscr)$, let $\Sigmah \subset \Gscr$ be a closed subset such that $\Sigmah \supset \supp \muh$. 
We regard $(\Sigmah, \muh)$ as a space of symbols and consider, as before, the shift $\sigma$ on the space $X := \Sigmah^\Z$ of sequences $\om = \left\{ \om_n \right\}_{n\in\Z}$ endowed with the product measure $\muh^Z$ and the product topology (which is metrizable). 
The standard projections 
\begin{align*}
\alfa \colon \Sigmah \to \T^d, & \qquad \alfa (\alpha, A) = \alpha \\
\Ascr \colon \Sigmah \to C^0 (\T^d, \SL_m (\R)), & \quad \quad \Ascr (\alpha, A) = A
\end{align*}
determine the linear cocycle $F=F_{(\alfa,\Ascr)} \colon X\times\T^d\times\R^m \to X\times\T^d\times\R^m$ defined by
$$ F(\omega,\theta, v) := \left(\sigma \omega, \theta+\alfa(\omega_0), \Ascr(\omega_0)(\theta)\, v \right)  .$$
The non-invertible version of this map, with the same expression, is defined on $\Xplus\times \T^d\times \R^m$, where $\Xplus=\Sigmah^{\N}$.

Thus the base dynamics of the cocycle $F$ is the mixed random-quasiperiodic map $f$ defined above,
$$ X \times \T^d \ni (\om, \theta) \mapsto f (\om, \theta) := \left(\sigma \om, \theta + \alfa (\om_0) \right) \in X \times \T^d ,$$
while the fiber action is induced by the map
$$X \times \T^d \ni ( \om, \theta ) \mapsto \Ascr (\om, \theta) := \Ascr (\om_0) (\theta) \in \SL_m (\R) .$$

The skew-product $F$ will then be referred to as a {\em mixed random-quasiperiodic cocycle}. The space of mixed cocycles $F = F_{(\alfa, \Ascr)}$ is a metric space with the uniform distance
$$
\dist \left( (\alfa,\Ascr), \, (\alfa',\Ascr') \right)= \norm{\alfa-\alfa'}_0 + \norm{\Ascr-\Ascr'}_0 \, .
$$

For $\om = \{ \om_n \}_{n\in\Z} \in X$ and $j \in \N$ consider the composition of random translations
\begin{align*}
\tau_\om^j & := \tau_{\alfa(\om_{j-1})} \circ \cdots \circ \tau_{\alfa(\om_0)} 
 = \tau_{ \alfa(\om_{j-1}) + \cdots +  \alfa(\om_0) } = \tau_{\alfa (   \om_{j-1} \circ \cdots \circ \, \om_0  ) } \, .
\end{align*}

The iterates of the cocycle $F$ are then given by
$$F^n (\om, \theta, v) = \left( \sigma^n \om, \tau_\om^n (\theta), \, \Ascr^n (\om) (\theta) v \right) ,$$
where
\begin{align*}
\Ascr^n (\om) & = \Ascr \left( \om_{n-1} \circ \cdots \circ \om_1 \circ \om_0 \right) \\
& = \left( \Ascr (\om_{n-1}) \circ \tau_\om^{n-2}  \right) \, \cdots \,  \left( \Ascr (\om_{1}) \circ \tau_\om^{0}  \right) \, \Ascr (\om_0) \, .
\end{align*}
Thus  $\Ascr^n(\omega)$ can be interpreted as a random product of
quasiperiodic cocycles, driven by the measure $\muh$ on the group $\Gscr$ of such cocycles. For convenience we also denote $\Ascr^n (\om, \theta) := \Ascr^n (\om) (\theta)$.

\medskip

By the subadditive ergodic theorem, the limit of $\displaystyle \frac{1}{n} \, \log \norm{ \Ascr^n (\om) (\theta)}$ as $n\to\infty$ exists for $\muh^\Z \times m$ a.e. $(\om, \theta) \in X \times \T^d$. Recall that for simplicity we call the measure $\muh$ ergodic if the base dynamics $f$ is ergodic w.r.t. $\muh^\Z \times m$. In this case, the limit is a constant that depends only on the measure $\muh$ and it is called the first (or maximal) Lyapunov exponent of the cocycle $F$, which we denote by $L_1 (\muh)$. Thus
\begin{align*}
L_1 (\muh) & = \lim_{n\to\infty} \frac{1}{n} \, \log \norm{ \Ascr^n (\om) (\theta)} \quad \text{for} \quad \muh^\Z \times m \text{ a.e. } (\om, \theta) \\
& = \lim_{n\to\infty} \, \int_{X \times \T^d}  \frac{1}{n} \, \log \norm{ \Ascr^n (\om) (\theta)} \, d ( \muh^\Z \times m ) \, .
\end{align*}

The second Lyapunov exponent $L_2 (\muh)$ is defined similarly, where instead of the norm (that is, the first singular value) of the iterates of the cocycle, we consider the second singular value.

\subsection*{The Schr\"{o}dinger cocycle} Given an ergodic translation $\alpha \in \T^d$, a potential function $v \in C^0 (\T^d, \R)$ and an energy $E \in \R$, consider the corresponding quasiperiodic Schr\"{o}dinger cocycle $(\alpha,S_E)$ where
$$
S_E (\theta)=\begin{bmatrix} v (\theta)-E & -1 \\ 1 & 0 \end{bmatrix}.
$$

This cocycle is related in the usual way to the discrete quasiperiodic Schr\"odinger operator
$$(H (\theta)  \psi)_n = - \psi_{n+1} - \psi_{n-1} +  v (\theta + n \alpha)  \, \psi_n \quad \forall n \in \Z .$$

For any real number $\om$, if we denote
$$P (\om) := \begin{bmatrix} 1 & \om  \\ 0 & 1 \end{bmatrix} \, $$
then
$$\Ascr_E (\om, \theta) := P (\om) \, S_E (\theta) = \begin{bmatrix} v (\theta) + \om - E & -1 \\ 1 & 0 \end{bmatrix} .$$

\medskip

Given a measure $\rho \in \Prob_c (\R)$ and an i.i.d. sequence $\{w_n\}_{n\in\Z}$ of real valued random variables with common law $\rho$, the mixed random-quasiperiodic cocycle corresponding to the Schr\"odinger operator~\eqref{op1} is then driven by the measure $\muh_E \in \Prob_c (\Gscr (d, 2) )$,
$$\muh_E := \delta_\alpha \times \int_\R \delta_{P (\om) S_E} \, d \rho (\om) \, .$$

As explained earlier, since $\alpha$ is an ergodic translation, the measures $\muh_E$ are automatically ergodic. Note that their supports are contained in the closed set 
$$\Sigmah := \{\alpha\} \times \left\{ P (\om) \, S_E \colon \om \in \supp (\rho), E \in \R \right\} \, .$$

\medskip

If instead we randomize the frequency, that is, given a probability $\mu \in \Prob (\T^d)$ and an i.i.d  sequence $\{\alpha_n\}_{n\in\Z}$ of frequencies with common law $\mu$, the mixed random-quasiperiodic cocycle corresponding to the Schr\"odinger operator~\eqref{op2} is driven by the measure $\muh_E \in \Prob_c (\Gscr (d, 2) )$,
$$\muh_E :=  \mu   \times  \delta_{S_E}  \, .$$

As explained earlier, the measure $\muh_E$ is ergodic if and only if for every $k \in \Z^d \setminus \{0\}$ there is $\alpha \in \supp (\mu)$ such that $\avg{k, \alpha} \notin \Z$. Note that the supports of the measures $\muh_E$ are contained in the closed set
$$\Sigmah := \supp (\mu) \times \left\{ S_E \colon E \in \R \right\} \, .$$

\medskip

We may of course choose to randomize both the frequency and the quasiperiodic cocycle:
given probability measures $\rho \in \Prob_c (\R)$ and 
$\mu \in \Prob (\T^d)$, the mixed random-quasiperiodic cocycle corresponding to the Schr\"odinger operator~\eqref{op3} is driven by the measure $\muh_E \in \Prob_c (\Gscr (d, 2) )$,
$$\muh_E:= \mu \times \int_\R \delta_{P (\om) S_E} \, d \rho (\om) \, .$$

Again, these measures are ergodic if and only if for every $k \in \Z^d \setminus \{0\}$ there is $\alpha \in \supp (\mu)$ such that $\avg{k, \alpha} \notin \Z$. Moreover, their supports are contained in the closed set 
$$\Sigmah := \supp (\mu) \times \left\{ P (\om) \, S_E \colon \om \in \supp (\rho), E \in \R \right\} \, .$$

\subsection*{The main statements} In Section~\ref{generic} we will derive the main results of Furstenberg's theory for the mixed random-quasiperiodic cocycles introduced above. Some of these results, like Furstenberg's formula and the continuity of the Lyapunov exponent for generic cocycles, had already been established in a setting that includes ours by Y. Kifer~\cite{Kifer-book}. 

One of the main results of this paper is the following criterion for the positivity of the maximal Lyapunov exponent.

\begin{theorem}\label{thm intro positivity}
Let $\muh\in \Prob_c(\Gscr)$ and assume that $\muh$ is ergodic, non-compact and strongly irreducible.
Then $\fle(\muh)>0$.
\end{theorem}

The concept of strongly irreducible measure on the group $\Gscr$ of quasiperiodic cocycles is rather technical, see Definition~\ref{non-compact/strongly irreducible, Zariski: SLm}. A more intuitive sufficient condition for strong irreducibility is the non-existence of a finite union of proper measurable sub-bundles which is invariant under $\muh$-a.e. element of the group (see Proposition~\ref{strongly irreducible}). This is the analogue of the classical concept of strong irreducibility in the group of matrices. 

Besides the classical Furstenberg's theorem, positivity criteria for the Lyapunov exponent were obtained in other, more abstract settings, see~\cite[Theorem 3.17]{FurmanSurvey} and~\cite[Theorem 6.1]{FurmanMonod}. However, they do not apply to our setting, as our group $\Gscr$ is not locally compact. 

This criterion is applicable to the cocycles corresponding to the Schr\"odinger operators~\eqref{op1},~\eqref{op2} and~\eqref{op3} and in Section~\ref{example} we derive the following consequences. 

\begin{proposition}\label{intro prop1}
Consider the Schr\"odinger operator~\eqref{op1} with randomly perturbed quasiperiodic potential and the corresponding cocycles driven by the measures
$\muh_E := \delta_\alpha \times \int_\R \delta_{P (\om) S_E} \, d \rho (\om)$
where $\rho \in \Prob_c (\R)$.

If $\supp (\rho)$ has more than one element, then $L_1 (\muh_E) > 0$ $\forall E \in \R$. Moreover, the map $E \mapsto L_1 (\muh_E)$ is continuous.
\end{proposition}

We actually derive a more general result, where the quasiperiodic potential is perturbed by random potentials, rather than numbers, see Proposition~\ref{ex prop2}.

We now consider the case of randomly chosen frequencies.

\begin{proposition}\label{intro prop1}
Consider the Schr\"odinger operator~\eqref{op2} with randomly perturbed frequencies and the corresponding cocycles driven by the measures $\muh_E :=  \mu   \times  \delta_{S_E} $
where $\mu \in \Prob (\T^d)$.

If there are two frequencies $\alpha, \beta \in \supp (\mu)$ such that $\beta - \alpha$ is an ergodic translation on $\T^d$, and if the potential function $v (\theta)$ is analytic and non-constant (or, more generally, if the continuous functions $v (\theta)$ and $v (\theta + \beta - \alpha)$ are transversal), then $L_1 (\muh_E) > 0$ $\forall E \in \R$. Moreover, the map $E \mapsto L_1 (\muh_E)$ is continuous.
\end{proposition}

\begin{proposition}\label{intro prop3}
Consider the Schr\"odinger operator~\eqref{op3} with randomly perturbed frequencies and quasiperiodic potential and the corresponding cocycles driven by the measures
$\muh_E := \mu \times \int_\R \delta_{P (\om) S_E} \, d \rho (\om)$
where $\mu \in \Prob (\T^d)$ and $\rho \in \Prob_c (\R)$.

If $\supp (\mu)$ contains an ergodic frequency and if $\supp (\rho)$ has more than one element, then $L_1 (\muh_E) > 0$ $\forall E \in \R$. Moreover, the map $E \mapsto L_1 (\muh_E)$ is continuous.
\end{proposition}

\begin{remark}
Depending on the regularity class, the Lyapunov exponent of quasiperiodic cocycles can have discontinuity points, as shown by Y. Wang and J. You~\cite{WangYou}. The continuity results above show that the addition of random noise has a regularizing effect on the Lyapunov exponent, thus making the mixed cocycles behave more like their random counterparts (for which the Lyapunov exponent is always continuous). 
\end{remark}

A result of Furstenberg's theory in our setting, related to Kifer's non-random filtration, shows that given an ergodic measure $\muh \in \Prob_c (\Gscr)$, if $L_1 (\muh) > L_2 (\muh)$ and $\muh$ is quasi-irreducible, then the convergence
$$\frac{1}{n} \, \log \norm{\Ascr^n (\om, \theta) p } \to L_1 (\muh)$$
 holds for $\muh^\Z$-a.e. $\om$ and for all $\theta \in \T^d$ and $p \in \R^m$ unitary vector.
 
 It turns out that taking the expected value in $\om$, the convergence above is uniform in $(\theta, p)$. This kind of result is available in the classical setting of products of random matrices, where the uniformity of the convergence is in the direction $p$.
 We obtain its analogue in the mixed setting, where the uniformity also holds in the quasiperiodic variable (which is not completely surprising, given the unique ergodicity of the torus translation).
 
 \begin{theorem}\label{intro thm2}
 Let $\muh \in \Prob_c (\Gscr)$ be an ergodic, quasi-irreducible measure. If $L_1 (\muh) > L_2 (\muh)$ then
 $$\EE \left(  \frac{1}{n} \, \log \norm{\Ascr^n (\theta) p } \right) \to L_1 (\muh) $$
 uniformly in $(\theta, p) \in \T^d \times \Su^{m-1}$. 
 \end{theorem}
 
 This result represents the main ingredient that will allow us (in a future project) to study the spectral properties of the Markov operator associated to the mixed cocycle, which in turn will be used to obtain statistical properties for the iterates of the cocycle.

\section{Stochastic dynamical systems}\label{generalities}
 We introduce some general concepts that will be used throughout the paper.

Let $M$ be a metric space. A stochastic dynamical system (SDS) on $M$ is any continuous map  $K \colon M\to\Prob (M)$, $x\mapsto K_x$.
An SDS $K$ on $M$ induces a bounded linear operator (called the Markov operator)
$\Qop_K \colon C_b (M)\to C_b (M)$ defined by
$$ (\Qop_K \varphi)(x):= \int_M \varphi(y)\, dK_x(y) .$$
It also induces the adjoint operator
$\Qop_K^\ast :\Prob (M)\to \Prob (M)$ of $\Qop_K$ characterized by
$$ \Qop_K^\ast \nu = K\ast \nu  := \int_M  K_x\, d\nu(x) . $$
A measure $\nu\in \Prob(M)$ is called \textit{$K$-stationary} if
$\Qop_K^\ast \nu=\nu$. We denote by $\Prob_K(M)$ the convex and compact (when $M$ is compact) subspace of
all $K$-stationary probability measures on $M$.

Let $(G, \cdot)$ be a topological group acting on $M$ from the left. Denote by $\tau_g \colon M \to M$ the action on $M$ by $g\in G$, that is, $\tau_g (x) = g x$.

Given $\mu \in \Prob_c (G)$ and $\nu \in \Prob_c (M)$, the convolution $\mu \ast \nu \in \Prob_c (M)$ is given by
$$\mu \ast \nu (E) := \int_M \int_G \ind_E (g x) \, d \mu (g) d \nu (x) $$
for any Borel set $E \subset M$.

Then
$$\mu \ast \nu = \int_G \left( \tau_g \right)_\ast \nu \, d \mu (g) \, ,$$
where $\left( \tau_g \right)_\ast \nu $ is the push-forward probability measure
$$\left( \tau_g \right)_\ast \nu (E) := \nu \left(\tau_g^{-1} E \right) = \nu \left (g^{-1} E \right) \, .$$

A probability measure $\mu \in \Prob_c (G)$ determines an SDS on $M$ by
$$M \ni x \mapsto \mu \ast \delta_x = \int_G \delta_{g x } \, d \mu (g) \in \Prob _c (M) \, .$$

The associated Markov operator $\Qop_\mu \colon C_b (M) \to C_b (M)$ is given by
$$\left( \Qop_\mu \phi \right) (x) = \int_M \phi (y) \, d \mu \ast \delta_x (y) = \int_G \phi (g x) \, d \mu (g) \, .$$

Moreover, its dual operator $\Qop_\mu^\ast \colon \Prob_c (M) \to \Prob_c (M)$ is
$$\Qop_\mu^\ast \nu = \int_M \mu \ast \delta_x \, d \nu (x) = \mu \ast \nu \, .$$

Let
$$\Prob_\mu (M) := \left\{ \nu \in \Prob_c (M) \colon  \mu \ast \nu = \nu \right\} $$
be the set of $\mu$-stationary measures on $M$, that is, the fixed points of the dual Markov operator $\Qop_\mu^\ast$.

Given such a $\mu$-stationary measure $\nu$, any observable $\phi \colon M \to \R$ for which
$$\left( \Qop_\mu \phi \right) (x) = \phi (x) \quad \text{ for } \nu \text{ a.e. } x \in M$$
is called a $\nu$-stationary observable.

\medskip

Here are some relevant examples that fit the abstract concepts introduced above.

The group of cocycles $\Gscr$ has several (left) actions on $\T^d$, $\T^d\times \R^m$ and $\T^d\times\Proj$, respectively defined as follows:
\begin{align*}
\Gscr\times \T^d \to\T^d,\quad &(\alpha, A)\cdot \theta := \theta +\alpha , \\
\Gscr\times \T^d\times \R^m \to\T^d\times \R^m,\quad &(\alpha, A)\cdot (\theta,v) := (\theta +\alpha, A(\theta)  v ) , \\
\Gscr\times \T^d\times \Proj \to\T^d\times \Proj,\quad &(\alpha, A)\cdot (\theta,\hat p ) := (\theta +\alpha, A(\theta) \hat p )  .
\end{align*}

In another setting, specializing to $G = M = \T^d$ seen as an additive group, for $\alpha, \theta \in \T^d$ and $\mu \in \Prob (\T^d)$ we have
$\tau_\alpha (\theta) = \theta + \alpha$,
$$\left( \Qop_\mu \phi \right) (\theta) = \int_{\T^d} \phi (\theta + \alpha) \, d \mu (\alpha) $$
and
$$ \Qop_\mu^\ast \nu (E) =    \int_{\T^d} \left( \tau_\alpha \right)_\ast \nu (E) \, d \mu (\alpha) = 
\int_{\T^d} \nu \left( \tau_\alpha^{-1} E \right)  \, d \mu (\alpha) $$
for any Borel measurable set $E \subset \T^d$.

Note that the Haar measure $m$ is $\mu$-stationary (since it is translation invariant).

\section{Generic cocycles}\label{generic}

In this section we introduce and relate several irreducibility concepts and derive the main results of Furstenberg's theory, including a representation formula for the maximal Lyapunov exponent, Kifer's non-random filtration, a criterion for the positivity of the maximal Lyapunov exponent and its continuity  in a generic setting.

\subsection*{Furstenberg-Kifer lemma }

Fix a probability measure  $\muh\in \Prob_c(\Gscr)$, where $\Gscr$ is the group of cocycles introduced in Section~\ref{intro} and consider the action of $\Gscr$ on the compact space $\QP := \T^d\times\Proj$. Let $\Sigmah\subset \Gscr$ be a compact set containing $\supp(\muh)$
and set $M:=\Sigmah\times \QP$. The measure $\muh$ determines, via the action  of $\Gscr$ on $\QP$,  the stochastic dynamical system $K=K_{\QP}:\QP\to \Prob(\QP)$ defined by:
\begin{equation}
\label{def KQ}
 K_{(\alpha, \hat p)}:= \muh\ast \delta_{(\alpha, \hat p)}  .
\end{equation}
These objects also determine $K=K_M:M\to \Prob(M)$, another SDS,
\begin{equation}
\label{def KM}
K_{((\alpha, A), \theta, \hat p)}:= \muh\times   \delta_{(\alpha, A)\cdot(\theta, \hat p)} =  \muh\times   \delta_{(\theta+\alpha, \hat A(\theta) \hat p)}    .
\end{equation}
Denote by $\Qop_\QP$ and $\Qop_M$ the Markov operators of the SDS $K_\QP$ and $K_M$. As before we  write
$\Prob_{\muh}(\QP)$ and $\Prob_{\muh}(M)$ instead of
$\Prob_{K}(\QP)$ and $\Prob_{K}(M)$, respectively. 

\begin{proposition}
\label{QQ factor of KM}
These operators are related by the following commutative diagram
$$ \begin{CD}
C^0(M)     @>\Qop_M>> C^0(M)\\
@V\Pi VV        @VV\Pi V\\
C^0(\QP)     @>\Qop_\QP>>  C^0(\QP)
\end{CD}
$$
where $\Pi \colon C^0(M)\to C^0(\QP)$ stands for the projection
$$(\Pi \varphi)(\theta,\hat p):= \int_\Sigmah \varphi(\omega, \theta,\hat p)\, d\muh(\omega)$$
for every $\varphi\in C^0(M)$.
\end{proposition}

\begin{proof} Given $\varphi\in C^0(M)$ a simple calculation gives
\begin{align*}
(\Pi\circ \Qop_M \, \varphi)(\theta, \hat p) &= \int_\Sigmah \int_\Sigmah \varphi(\omega, \theta+\alpha, A(\theta)\, \hat p)\, d\muh(\omega)\, d\muh(\alpha, A) \\
&= (\Qop_\QP\circ \Pi\, \varphi)(\theta, \hat p)  .
\end{align*}
\end{proof}

\begin{proposition}
\label{nu vs muh x nu}
If $\eta\in \Prob(\QP)$ is a $K_\QP$-stationary measure then
$\muh\times \eta$ is a $K_M$-stationary measure. Conversely, given a
$K_M$-stationary measure $\etah\in \Prob(M)$, there exists
a $K_\QP$-stationary measure $\eta\in \Prob(\QP)$ such that $\etah = \muh\times \eta$.
\end{proposition}

 \begin{proof}
A simple calculation gives that for any $\eta\in \Prob(\QP)$
\begin{equation}
\label{QM* and QQ*}
\Qop_M^\ast (\muh\times \eta) = \muh\times (\Qop_K^\ast \eta) .
\end{equation}
This fact follows also from Proposition~\ref{QQ factor of KM}.
Interpreting a measure $\eta$ as a linear functional
$\eta(\varphi):=\int \varphi\, d\eta$ we have $\eta \circ \Pi= \muh\times \eta$,
or equivalently $\Pi^\ast \eta= \muh\times\eta$. Hence taking adjoints, the relation
$\Pi\circ \Qop_M=\Qop_\QP\circ \Pi$  translates to
$\Qop_M^\ast \circ \Pi^\ast = \Pi^\ast \circ \Qop_\QP^\ast$ which by the previous remark implies~\eqref{QM* and QQ*}.

If $\eta$ is $K_\QP$-stationary, i.e., $\eta=\Qop_\QP^\ast \eta$,
then $\muh\times \eta= \Qop_M^\ast(\muh \times \eta)$
which implies that $\muh\times \eta$ is $K_M$-stationary.
Conversely, because the set $\Hscr=\{\muh\times\nu \colon \nu\in\Prob(\QP)\}$
is compact (and hence closed) in $\Prob(M)$, then
$$ \Qop_M^\ast  \etah = \int\underbrace{ \muh\times \delta_{(\theta+\alpha , \hat A(\theta)\hat p)} }_{\in \Hscr} \,
d\etah ((\alpha, A), \theta, \hat p) \in \Hscr $$
which proves that $\Qop_M^\ast \etah = \muh \times \eta$ for some $\eta\in \Prob(\QP)$. Hence
$$ \muh\times \eta = \Qop_M^\ast \etah = \etah  $$
and  by~\eqref{QM* and QQ*}
$$\muh\times \eta  = \etah   = \Qop_M^\ast \etah= \Qop_M^\ast (\muh\times \eta)
=\muh \times \Qop_\QP^\ast \eta $$
which implies that $\eta = \Qop_\QP^\ast \eta$, i.e., $\eta$ is $K_\QP$-stationary.
\end{proof}

\bigskip

A $K$-Markov process is any $M$-valued process
$\{ Z_n:\Omega\to M\colon n\geq 0\}$ on some probability space
$(\Omega,\FF,\Pp)$ such that for every Borel set $B\subset M$ and all $n\geq 1$,
$$ \Pp[\, Z_n\in B\, \vert \, Z_0, Z_1,\ldots, Z_{n-1}\, ] =K_{Z_{n-1}}(B) .$$
 $K$ rules the transition probabilities of the $K$-Markov processes.
Given a $K$-Markov process $\{Z_n\}_{n\geq 0}$,  the following are equivalent:
	\begin{enumerate}
		\item[(a)] The process $Z_n$ is stationary;
		\item[(b)]  The distribution of $Z_0$ on $M$
		is a $K$-stationary measure.
	\end{enumerate}

We recall the following results from H. Furstenberg anf Y. Kifer.

\begin{theorem}[Furstenberg-Kifer~\cite{FKi}{Theorem 1.1}]
	\label{Furstenberg Kifer Markov Ergodic Theorem}
	Let $\{Z_n\}_{n\geq 0}$ be a $K$-Markov process and $f\in\Cscr(M)$.
	Then with probability one,
	$$ \limsup_{n\to +\infty}  \frac{1}{n}\,\sum_{j=0}^{n-1} f(Z_j) \leq
	\sup \left\{\, \int f\, d\eta \, \colon \eta\in \Prob_{\muh}(M)   \right\} .$$
\end{theorem}

\begin{theorem}[Furstenberg-Kifer~\cite{FKi}{Theorem 1.4}]
\label{Furstenberg Kifer Markov Ergodic Corollary}
Let $\{Z_n\}_{n\geq 0}$ be a $K$-Markov process and $f\in\Cscr(M)$ and assume
that for every $K$-stationary probability measure $\eta\in \Prob_{\muh}(M)$,
$\int_{M} f\, d\eta= \beta$.
Then with probability one,
$$ \lim_{n\to +\infty}  \frac{1}{n}\,\sum_{j=0}^{n-1} f(Z_j) =\beta .$$
\end{theorem}

 \bigskip

\subsection*{Furstenberg's formula}

%% Furstenberg's Formula
Consider the one and two sided Bernoulli shifts $\sigma \colon \Xplus\to \Xplus$, resp. $\sigma \colon X\to X$, on the spaces of sequences $\Xplus=\Sigmah^\N$ and $X=\Sigmah^\Z$, respectively endowed with the product probability measures
$\muh^\N$ and  $\muh^\Z$. Denote by  $\PF \colon \Xplus\times \QP \to \Xplus\times \QP$ and  $\PFh \colon X\times \QP \to X\times \QP$
the corresponding projective  maps induced by $F$,
$$\PFh(\omega, \theta, \hat p):=( \sigma\omega, \omega_0\cdot(\theta, \hat p)) = ( \sigma\omega,  \theta+\alpha_0, \hat A_0(\theta) \hat p)) $$
where $(\alpha_0,A_0)=\omega_0$ is the $0$-th coordinate of $\omega$.

\begin{proposition}
	\label{stationary vs invariant}
	Given $\eta\in \Prob(\QP)$ the following are equivalent:
	\begin{enumerate}
		\item $\muh\times \eta$ is $K_M$-stationary;
		\item $\eta$ is $K_\QP$-stationary;
		\item $\muh^\N \times \eta$ is $\PF$-invariant.
	\end{enumerate}
\end{proposition}

\begin{proof}
The equivalence  $(1)\Leftrightarrow(2)$  was proved in Proposition~\ref{nu vs muh x nu}. See~\cite[Proposition 5.5]{Viana-book} for the proof of the equivalence $(1)\Leftrightarrow(3)$.
\end{proof}

Recall that $\alfa:\Gscr\to \T^d$ is the projection
$\alfa(\beta, B):=\beta$. Define also
$\pi:\QP\to\T^d$, $\pi(\theta, \hat p):=\theta$.

\begin{proposition}
Given $\muh\in\Prob_c(\Gscr)$ such that $\mu=\alfa_\ast \muh$ is ergodic, if $\eta\in \Prob_{\muh}(\QP)$ then $\pi_\ast\eta=m$.
\end{proposition}

\begin{proof}
The group $\Gscr$ acts on both $\QP$ and $\T^d$. Notice the action on $\T^d$ is given by $(\beta, B)\cdot \theta := \alfa(\beta, B)\cdot \theta=\theta+\beta$. The projection map $\pi:\QP\to \T^d$ is $\Gscr$-equivariant. Moreover, for all $(\beta, B)\in\Gscr$ and $(\theta, \hat p)\in \QP$,
$$ \alfa(\beta, B)\cdot \pi(\theta,\hat p) = \pi((\beta, B)\cdot (\theta, \hat p)) . $$
Hence, setting $\mu:=\alfa_\ast\muh$ and taking $(\beta, B)$ with distribution $\muh$, and $(\theta,\hat p)$ with distribution $\eta$ we get that
\begin{align*}
\mu \ast (\pi_\ast\eta)&= (\alfa_\ast\muh)\ast \pi_\ast\eta  = \pi_\ast( \muh\ast \eta) =\pi_\ast \eta .
\end{align*}
Therefore, since $\mu$ is ergodic, by~\cite[Theorem 2.3 item (7)]{CDK-paper1} we have that $\pi_\ast \eta\in \Prob_{\mu}(\T^d)=\{m\}$, which finishes the proof.
\end{proof}

Using Rokhlin disintegration theorem (see for instance~\cite[Theorem 5.1.11]{Viana-Oliveira-ET-en}) we can talk about the disintegration of the measure $\eta$ with respect to $\pi:\QP\to \T^d$, which will be denoted by $\{\eta_\theta\}_{\theta\in\T^d}$. By the uniqueness of Rokhlin disintegration,  the disintegrated measures  $\eta_\theta$ are well defined for $m$ almost every $\theta$.

\begin{proposition}
\label{invariance disintegration}
Given $\eta\in \Prob(\QP)$ such that $\pi_\ast \eta=m$, the measure
$\eta$ is $\muh$-stationary if and only if its disintegration $\{\eta_\theta\}_{\theta\in\T^d}$ satisfies
\begin{equation*}
\eta_\theta = \int B(\theta-\beta)_\ast \eta_{\theta-\beta}\, d\muh(\beta, B) \quad \text{ for m-a.e. } \theta\in\T^d .
\end{equation*}
\end{proposition}

\begin{proof}
Straightforward verification.
\end{proof}

\begin{definition}
	Given $\eta\in \Prob_{\muh}(\QP)$, a measurable function $\varphi\in L^\infty(\QP)$ is said to be $\eta$-stationary if \, $\Qop_\QP\varphi=\varphi$\, $\eta$-a.e.. Likewise, a Borel set $A\subset \QP$
	is said to be $\eta$-stationary if its indicator function $\ind_A$ is
	$\eta$-stationary.
\end{definition}

\begin{definition}
	A stationary measure $\eta\in \Prob_{\muh}(\QP)$ is called  ergodic if
	for every $\eta$-stationary Borel set $A\subset \QP$, either $\eta(A)=0$ or  $\eta(A)=1$.
\end{definition}

\begin{proposition}
\label{extremal vs ergodic}
Given $\eta\in \Prob_{\muh}(\QP)$ the following are equivalent:
\begin{enumerate}
	\item $\eta$ is an extremal point of $\Prob_{\muh}(\QP)$;
	%\item $\muh\times \eta$ is an extremal point of $\Prob_{\muh}(M)$;
	\item  $\eta$ is an ergodic stationary measure;
	\item The map $(\PF, \muh^\N \times \eta)$ is ergodic.
\end{enumerate}
\end{proposition}

\begin{proof}
See the proof of the equivalence $(2)\Leftrightarrow(3)$ in~\cite[Proposition 5.13]{Viana-book}.

\smallskip

\noindent
We prove now that $(1)\, \Rightarrow\, (2)$.
If (2) does not hold there exists a $\eta$-invariant Borel  set $A\subset \QP$  with $0<\eta(A)<1$.
Define the measure $\eta_A\in\Prob(\QP)$,
$$ \eta_A(E):=\frac{1}{\eta(A)}\eta(A\cap E) \qquad \forall    E\in\B(\QP) .$$
To simplify notations we introduce the variables $x=(\theta,\hat p)\in \QP$ and
$\omega=(\alpha, A) \in \Sigmah$. With this notation it follows from the
identity
$$ \ind_A(x) = \int \ind_A(\omega\cdot x)\, d\muh(\omega)  \quad \text{ for } \eta\text{-a.e. } x\in \QP $$
and because the indicator function $\ind_A(x)$ can only take the values $0$ and $1$, that for $\eta$-a.e.  $x\in \QP$,
\begin{enumerate}
	\item  $x\in A$\; $\Rightarrow$ \;
	$\omega\cdot x\in A$\;  for $\muh$-a.e. $\omega\in \Sigmah$,
	
	\item $x\notin A$\; $\Rightarrow$ \;
	$\omega\cdot x\notin A$ \; for $\muh$-a.e. $\omega\in \Sigmah$.
\end{enumerate}
Thus $\Qop=\Qop_{\QP}$ satisfies for any given $\varphi\in C^0(\QP)$,
\begin{align*}
\int_{ \QP} \varphi\, d(\Qop^\ast \eta_A) &=
\int_{ \QP} (\Qop\varphi)\, d\eta_A =
\frac{1}{\eta(A)}\int_\QP\int_\Sigmah \varphi(\omega\cdot x)\,\ind_A(x)\, d\muh(\omega)\, d\eta(x)\\
&=
\frac{1}{\eta(A)}\int_ \QP \int_\Sigmah \varphi(\omega\cdot x)\,\ind_A(\omega\cdot x)\, d\muh(\omega)\, d\eta(x)\\
&=
\frac{1}{\eta(A)} \int_\QP  \Qop (\varphi \,\ind_A)\, d\eta = \frac{1}{\eta(A)} \int_\QP  \varphi \,\ind_A\, d(\Qop^\ast \eta)\\
&= \frac{1}{\eta(A)} \int_\QP \varphi \,\ind_A\, d\eta
= \int_\QP \varphi\, d\eta_A
\end{align*}
which proves that $\eta_A$ is $K$-stationary.
The complement set $B= \QP\setminus A$ is also $\eta$-stationary with $\eta(B)=1-\eta(A)\in ]0,1[$. Hence, by the same argument, $\eta_B$ is $K$-stationary. Therefore
$$ \eta=\eta(A)\, \eta_A + (1-\eta(A))\, \eta_B $$
is not an extremal point of $\Prob_{\muh}(\QP)$.

\smallskip

\noindent
Finally we prove that $(3)\, \Rightarrow\, (1)$.
If $\eta$ is not extremal then
$\eta= t\, \eta_1 + (1-t) \, \eta_2$ with
$\eta_1,\eta_2\in\Prob_{\muh}(\QP)$, $\eta_1\neq \eta_2$  and $0<t<1$.
This implies that
$$ \muh^\N\times\eta= t\, (\muh^\N\times\eta_1) + (1-t) \, (\muh^\N\times\eta_2) . $$
Since by Proposition~\ref{stationary vs invariant}
the probability measures $\muh^\N\times\eta_i$ are $\PF$-invariant, it follows that
$(\PF, \muh^\N\times \eta)$ is not ergodic.
\end{proof}

\bigskip

Consider the observable $\Psi \in C^0(M)$
defined by
\begin{equation}
\label{def Psi}
	 \Psi((\alpha,A),\theta, \hat p) = 	 \Psi_{ (\alpha,A)}(\theta, \hat p) := \log \norm{A(\theta) p} ,
\end{equation}
where $p$ is any unit vector representing the projective point $\hat p$,
and the continuous linear functional
$\alpha_{\muh}:\Prob_{\muh}(\QP)\to \R$ defined by
\begin{equation}
\label{def alpha muh}
 \alpha_{\muh}(\eta):= \int_\Sigmah \int_{\QP} \Psi_\omega(\theta, \hat p)\, d\eta(\theta, \hat p)\, d\muh(\omega).
\end{equation}
Because the convex space $\Prob_{\muh}(\QP)$ is weak-$\ast$ compact and $\alpha_{\muh}$ is continuous w.r.t. this topology we can define
$$ \beta(\muh):= \max\{ \alpha_\muh(\eta)\colon \eta \in \Prob_\mu(\QP)\, \} .$$

\smallskip

A simple adaptation of the argument~\cite[Theorem 2.1]{FKi}
establishes the following Furstenberg type formula.

\begin{theorem}
\label{Furstenberg-Kifer}
	Given $\muh\in\Prob_c(\Gscr)$ ergodic the following hold.
	
	\begin{enumerate}
		\item[(1)] \; $\displaystyle \lim_{n\to +\infty} \frac{1}{n} \log \norm{\Ascr^n(\omega)(\theta)} =\beta(\muh)$ \, for $m$-a.e.  $\theta\in\T^d$ and  $\muh^\N$-a.e. $\omega\in\Sigmah^\N$.
In particular, $\fle(\muh)=\beta(\muh)$.		
		\item[(2)] \; $\displaystyle \limsup_{n\to +\infty} \,  \frac{1}{n}\log \norm{\Ascr^n(\omega)(\theta)\, v} \leq \beta(\muh)$, for all $(\theta,v) \in \T^d\times\R^m$ and  $\muh^\N$-a.e. $\omega\in\Sigmah^\N$.
		\item[(3)] Assuming that the functional $\alpha_\muh$ is constant, we have \\
		\; $\displaystyle \lim_{n\to +\infty} \,  \frac{1}{n}\log \norm{\Ascr^n(\omega)(\theta)\, v} = \beta(\muh)$ for all $(\theta,v) \in \T^d\times \R^m\setminus\{0\}$ and $\mu^\N$-a.e. $\omega\in\Sigmah^\N$.	
	\end{enumerate}
\end{theorem}

\begin{proof}
	Let $K:M\to\Prob(M)$ be the SDS~\eqref{def KM}.
	For each $x=(\theta,\hat p)\in \T^d\times \Proj$,
	consider the random process $Z_n^x:\Xplus\to M$, $n\geq 0$, defined by
	$$
\begin{aligned}
Z_n^{x}(\omega):= &\left(  (\sigma^n \omega)_0, \theta+\tau^n(\omega),\hat \Ascr^n(\omega)(\theta)\hat p \right)\\
 = &\left(  \omega_n, \theta+\tau^n(\omega),\hat \Ascr^n(\omega)(\theta)\hat p \right).
\end{aligned}
$$
	
	The process $Z_n^{x}$ is a (non-stationary) $K$-Markov process on $(\Xplus,\Pp_{x})$ for a unique  measure  $\Pp_{x}\in\Prob(\Xplus)$ such that $\Pp_{x}[Z^x_0=x]= 1$,  the Kolomogorov extension determined by $(K,\delta_x)$.

	Next consider the continuous observable $\Psi\in\Cscr(M)$ defined in~\eqref{def Psi}.

	Notice that for any
	$x=(\theta, \hat p)$,	we have
	\begin{equation}
	\label{Psi and Ln(A)}
	 \frac{1}{n}\, \log \norm{\Ascr^n(\omega)(\theta) p} =  \frac{1}{n} \sum_{j=0}^{n-1} \Psi(Z_j^{x})  .
	\end{equation}
	Moreover, for any probability $\muh\times\eta\in\Prob_{\muh}(M)$ we have
	\begin{equation}
	\label{Psi average L(A)}
	\alpha_{\muh}(\eta) = \int \Psi\, d(\muh\times\eta) .
	\end{equation}
	Thus item (2) follows from Theorem~\ref{Furstenberg Kifer Markov Ergodic Theorem} applied to the SDS $K$ and the observable $f=\Psi$.
	
	Similarly, item (3) follows by Theorem~\ref{Furstenberg Kifer Markov Ergodic Corollary}.
	
	Let us now prove (1). By Furstenberg-Kesten's Theorem~\cite{FK60}
	the following limit exists for  $(\muh^\N\times m)$-a.e. $(\omega,\theta)\in\Sigmah^\N\times\T^d$
	$$ \lim_{n\to +\infty} \log \norm{\Ascr^n(\omega)(\theta)} =\fle(\muh)
	.$$

	Fixing a basis $\{e_1,\ldots, e_m\}$ of $\R^m$, define the matrix norm
	$$ \norm{g}' = \max_{1\leq i\leq m} \norm{g \,e_i} .$$
	By norm equivalence there exists $0<c<\infty$ such that
	$$  \frac{1}{c}\leq \frac{\norm{g}}{\norm{g}'}\leq c \quad \forall g\in\SL_m (\R). $$
	The set of maximizing measures
	$$\Mscr:=\{\eta\in \Prob_{\muh}(\QP)\colon \alpha_\muh(\eta)=\beta(\muh)\}$$
	is a non-empty compact convex set. By  Krein-Milman Theorem there exists an extremal point $\eta$ of $\Mscr$ and since $\Mscr$ is an extremal level set of the linear functional $\alpha_\muh$,
	this measure $\eta$ is also an extremal point of $\Prob_{\muh}(\QP)$.
	By Proposition~\ref {extremal vs ergodic},  $(\PF,\muh^\N\times\eta)$ is ergodic.
	Thus by~\eqref{Psi and Ln(A)},~\eqref{Psi average L(A)} and Birkhoff Ergodic Theorem, for $(\muh^\N\times \eta)$-a.e. $(\omega,\theta,\hat p) \in \Sigmah^\N\times \QP$,
	\begin{align*}
	\beta(\muh)=\alpha_\muh(\eta)& = \int_{M} \Psi \, d(\muh\times\eta) =   \int_{\Sigmah^\N\times \QP}\tilde \Psi \, d(\muh^\N\times\eta)\\
	&=\lim_{n\to +\infty} \frac{1}{n}\,\sum_{j=0}^{n-1} \tilde \Psi(\PF^j(\omega, \hat p)) = \lim_{n\to +\infty} \frac{1}{n}\,  \log \norm{\Ascr^n(\omega)(\theta)\, p}\\
	&\leq \lim_{n\to +\infty}  \frac{1}{n}\,  \log \norm{\Ascr^n(\omega)(\theta)}  = \lim_{n\to +\infty} \frac{1}{n}\,  \log \norm{\Ascr^n(\omega)(\theta)}'\\
	&= \max_{1\leq i\leq m}
	\lim_{n\to +\infty} \frac{1}{n}\,  \log \norm{\Ascr^n(\omega)(\theta)\, e_i} \leq  \beta(\muh) .
	\end{align*}
	
Here $\tilde \Psi := \Psi \circ \pi$ where $\pi: \Sigma^\N\times \QP \to \Sigma\times\QP$. This proves (1).
\end{proof}

%To relate $F$  with  $\Qop_\mu$
%we  use the notation\,
%$\EE_x[\,\varphi\, ]:= \int_X \varphi(\omega, x)\,d\mu^\Z(\omega)$.
%For every $n\in\N$ and $x\in M$ we have
%\begin{equation}
%(\Qop_\mu^n \varphi)(x) = \EE_x[	\,\varphi\circ F^n \,] .
%\end{equation}

\bigskip

\subsection*{ Kifer non-random filtration}

\begin{definition} Given $0\leq k \leq m$,
	we call $k$-dimensional measurable subbundle of $\T^d\times \R^m$
	any measurable subset $\Lscr\subset \T^d\times \R^m$ such that
	  $\Lscr_\theta:=\{ v\in \R^m \colon (\theta,v)\in \Lscr\}$ is a vector subspace with dimension  $k$ for  every $\theta\in\T^d$.
	
	The measurable subbundle $\Lscr$ is called $\muh$-invariant if
    $$B\,\Lscr_\theta = \Lscr_{\theta+\beta}$$
    for every $\muh$-almost every  $(\beta,B)\in \Gscr$ and $m$-almost every $\theta\in\T^d$. We express this invariance writing
    $(\beta, B)\Lscr=\Lscr$ for $\muh$-a.e. $(\beta,B)$.
\end{definition}

%\begin{proposition}
%If $L\subset \QP$ is a $(\mu,k)$-invariant subspace
%then
%$$\tilde L:= \{ (\theta,v)\in\T^d\times\R^m \colon
%(\theta,\hat v)\in L \} $$
%is an $F$-invariant measurable, linear, $k$-dimensional sub-bundle  of the trivial bundle with fiber $\R^m$ and basis $\T^d$. The restriction
%$F\vert_{\tilde L}\colon \tilde L\to \tilde L$ is a linear cocycle on the measurable vector bundle $\tilde L$.
%\end{proposition}

Denote by $\lambda_\muh(\Lscr)$ the first Lyapunov exponent of
$F\vert_{\Lscr}\colon\Lscr\to \Lscr$, i.e., the $\muh^\Z$-almost sure limit
$$\lambda_\muh(\Lscr)=\lim_{n\to +\infty} \frac{1}{n}\, \log \norm{\Ascr^n(\omega)(\theta) \vert_{\Lscr_\theta} } .$$

\bigskip

In our specific context the following result is a reformulation of~\cite[Chapter 3, Theorem 1.2]{Kifer-book}.

\begin{theorem}
\label{NonRandomFiltrarion}
Given $\muh\in\Prob_c(\Gscr)$, if $f$ is  ergodic w.r.t. $\muh^\Z\times m$ then there exists a sequence of invariant measurable subbundles
$$ \T^d\times\{0\}=\Lscr_{r+1}  \subset\Lscr_r \subset \Lscr_{r-1} \subset \cdots \subset \Lscr_1
\subset \Lscr_0=\T^d\times \R^m , $$
 a sequence of numbers $\beta_0> \beta_1 > \cdots > \beta_r>-\infty$
and a Borel set $B\subset \T^d$ with $m(B)=1$ such that for all $\theta\in B$,
if $(\theta, p)\in \Lscr_i \setminus \Lscr_{i+1}$, $1\leq i\leq r$,
then with probability one
$$ \lim_{n\to +\infty} \frac{1}{n}\, \log \norm{\Ascr^n(\theta)\, p} = \beta_i . $$
Moreover, the numbers $\beta_i$ with $0\leq i\leq r$  are exactly the values
  of the functional $\alpha_\muh$ on the extremal points of $\Prob_\muh(\QP)$.
\end{theorem}

\bigskip

\subsection*{Irreducibility concepts}

\begin{definition}
	\label{irred & quasi-irred}
We say that $\muh$ is irreducible if there exists
no proper $\muh$-invariant measurable   sub-bundle $\Lscr\subset \T^d\times\R^m$.

We say that $\muh$ is quasi-irreducible if there exists
no invariant measurable sub-bundle $\Lscr\subset \T^d\times\R^m$ such that
$\lambda_\muh(\Lscr)<\fle(\muh)$.
\end{definition}

\begin{proposition}
\label{quasi-irred charact}
The following are equivalent:
\begin{enumerate}
\item $\muh$ is quasi-irreducible;
\item The non-random filtration of $\mu$ in Theorem~\ref{NonRandomFiltrarion} is trivial, i.e.,
it consists of the single measurable sub-bundle $\Lscr_0=\T^d\times\R^m$;
\item $\alpha_\muh(\eta)=\fle(\muh)$ for every $\muh$-stationary measure $\eta\in\Prob(\QP)$;
\item  For every $(\theta,\hat p)\in \QP$,
with probability one,
$$
\lim_{n\to +\infty} \frac{1}{n}\,\log\norm{\Ascr^n(\theta)\,p } = \fle(\muh) .
$$
\end{enumerate}
\end{proposition}

\begin{proof}

$(1) \, \Rightarrow\, (2)$:
If (2) does not hold there exists a non trivial non random
invariant filtration $ \T^d\times\{0\} \subset\Lscr_r \subset \cdots \subset  \Lscr_0=\T^d\times \R^m $ with $r\geq 1$.
Then $\Lscr_1$ is an invariant measurable sub-bundle with $\lambda_\muh(\Lscr_1)=\beta_1<\beta_0=\fle(\muh)$, which shows that $\muh$ is not
quasi-irreducible.

$(2) \, \Rightarrow\, (3)$: By Theorem~\ref{NonRandomFiltrarion}, $\alpha_\muh(\eta)=\beta_0=\fle(\muh)$  for every  extremal point $\eta$ of $\Prob_{\muh}(\QP)$. Since every measure $\eta\in \Prob_{\muh}(\QP)$
is an average of extremal points, it follows that $\alpha_\muh(\eta)=\fle(\muh)$ as well.

$(3) \, \Rightarrow\, (4)$: By (3) the linear functional $\alpha_\muh$ is constant and equal to $\fle(\muh)$. Hence, by   Theorem~\ref{Furstenberg-Kifer} (3), conclusion (4) holds.

$(4) \, \Rightarrow\, (1)$: If $\muh$ is not quasi-irreducible and
$(\theta,v)\in \Lscr_1$ with $v\neq 0$ then by Theorem~\ref{NonRandomFiltrarion},
with probability one
$$ \lim_{n\to +\infty} \frac{1}{n}\, \log \norm{\Ascr^n(\theta)\, p} < \beta_0=\fle(\muh)   . $$
This contradicts (4).
\end{proof}

\begin{corollary}\label{corfle}
For any $\theta \in \T^d$,
$
\displaystyle \frac{1}{n} \,\log\norm{\Ascr^n(\omega,\theta)}$ converges in measure to $\fle(\muh).$
\end{corollary}
\begin{proof}
By item (4) of Proposition \ref{quasi-irred charact}, we have for any $\theta\in \T^d$ and any $\epsilon>0$ $$\lim_{n\to\infty}\muh^\Z\{ \omega\in X \colon  \frac{1}{n}\,\log\norm{\Ascr^n(\omega, \theta)}\leq \fle(\muh)-\epsilon\}=0.$$ Therefore, it is enough to prove for any $\theta\in \T^d$ and any $\epsilon>0$,
$$\lim_{n\to\infty}\muh^\Z\{  \omega\in X \colon  \frac{1}{n}\,\log\norm{\Ascr^n(\omega, \theta)}\geq \fle(\muh)+\epsilon\}=0 ,$$
which follows from~\cite[Theorem 3.1]{CDK-paper1}. 
\end{proof}

%\bigskip

%\blue{Definitions: strong-irreducibility, contraction, their genericity}

\begin{definition}[See Definition 1.6 in~\cite{FurmanSurvey}]
	\label{non-compact/strongly irreducible, Zariski: SLm}
Let $G\subset \SL_m(\R)$ be a closed subgroup.
We say that $G$ is
\begin{enumerate}
	\item   {\em reducible} if there exists a proper subspace $0\neq L\subsetneq \R^m$ such that $g\, L=L$ for all $g\in G$.
	
	\item {\em  virtually reducible }  if there exists a reducible subgroup $H\subset G$ with finite index, $[G:H]:=\#(G/H)<\infty$.
	
	\item   {\em strongly irreducible} if it is not virtually reducible.
	
	\item   {\em Zariski dense} if it is not contained in any proper algebraic subgroup of $\SL_m(\R)$.
\end{enumerate}

\end{definition}

For groups of cocycles, we define:

\begin{definition}[See Definition 3.16 in~\cite{FurmanSurvey}]
\label{non-compact/strongly irreducible, Zariski: Gscr}

Let $\Gscr_0\subset \Gscr$ be a closed subgroup.
We say that $\Gscr_0$ is respectively \textit{non-compact},
	\textit{strongly irreducible} or \textit{Zariski dense} when  there exists no measurable function $M\colon \T^d\to \SL_m(\R)$ such that
	$C:\Gscr_0\times\T^d\to \SL_m(\R)$,
	$$  C((\beta, B), \theta):= M(\theta+\beta)^{-1}\, B(\theta) \,M(\theta) $$
for any $(\beta, B)\in \Gscr_0$ and $m$-a.e. $\theta$, takes values in a proper closed subgroup  $G_0\subseteq \SL_m(\R)$ that is compact,   virtually reducible or  algebraic, respectively.

	Given measure $\muh\in\Prob_c(\Gscr)$ we denote by $\Gscr_\muh$ the closed subgroup of $\Gscr$ generated by $\supp(\muh)$.
	This measure $\muh$ is called  \textit{non-compact},
	\textit{strongly irreducible} or \textit{Zariski} when
	the closed subgroup $\Gscr_\mu$ satisfies the same properties.
\end{definition}

\begin{proposition}
	\label{strongly irreducible}
Let $\muh\in\Prob_c(\Gscr)$. Then 
\begin{enumerate}
	\item $\muh$
	is \textit{irreducible} if and only if there exists no measurable function  $M:\T^d\to \SL_m(\R)$ and a reducible subgroup   $G_0\subseteq \SL_m(\R)$ such that
	$M(\theta+\beta)^{-1}\, B(\theta) \,M(\theta)\in G_0$ for all $(\beta, B)\in \Gscr_{\muh}$.
	
	\item  if there exist no finite union $\cup_{j=1}^\ell  \Lscr_j$ of   proper measurable   sub-bundles  $\Lscr_j$, $j=1,\ldots, \ell$,  such that
	$$\cup_{j=1}^\ell  (\beta, B) \Lscr_{j}   = \cup_{j=1}^\ell  \Lscr_{j}$$
	for every $\muh$-a.e. $(\beta,B)\in \Gscr$
	then $\muh$
	is  strongly irreducible.
\end{enumerate}
\end{proposition}

\begin{proof}
  If $\muh$ is  reducible  there exists
a $k$-dimensional  $\muh$-invariant measurable   sub-bundle $\Lscr\subset \T^d\times\R^m$, for some $1\leq k\leq m-1$.
Take a measurable function $M:\T^d\to \SL_m(\R)$ such that the first $k$ columns of $M(\theta)$ form a basis of $\Lscr_\theta$.
The subgroup 
$$G_0 :=\{ g\in \SL_m(\R)\colon \, g\,(\R^k\times\{0\} ) = \R^k\times\{0\} \, \} $$
is a reducible group. By construction,
$M(\theta+ \beta)^{-1}\, B(\theta)\, M(\theta)\in G_0$ for
a.e. $\theta\in\T^d$ and every $(\beta,B)\in \Gscr_\muh$.

Conversely, if there exists a measurable function  $M:\T^d\to \SL_m(\R)$ and a reducible subgroup   $G_0\subseteq \SL_m(\R)$ such that
$M(\theta+\beta)^{-1}\, B(\theta) \,M(\theta)$ takes values in $G_0$, for a.e. $\theta$ and  all $(\beta, B)\in \Gscr_{\muh}$, take a proper subspace $L\subset \R^m$ and define the $\muh$-invariant measurable sub-bundle
$\Lscr:=\{ (\theta, M(\theta)\, v) \colon  \theta\in\T^d,\; v\in L \}$. This finishes the proof of~ (1).

	\bigskip

 	If $\muh$ is not strongly irreducible,  there exists a measurable function $M\colon \T^d\to \SL_m(\R)$ such that
	$C:\Gscr\times\T^d\to \SL_m(\R)$,
	$$  C((\beta, B), \theta):= M(\theta+\beta)^{-1}\, B(\theta) \,M(\theta) , $$
	takes values in a proper and  closed, virtually reducible subgroup  $G\subseteq \SL_m(\R)$. This means there is  a reducible subgroup $H\subset G$ such that $\ell:=[G:H]<\infty$.
	Let $L\subset \R^m$ be a proper subspace such that $g\,L=L$ for all $g\in H$. Take elements $g_1,\ldots, g_{\ell}\in G$ for which  $G/H= \{g_1\,H, \ldots, g_{\ell}\, H\}$.
	Then the proper linear subspaces $g_1\,L, \ldots, g_{\ell}\, L$ satisfy
	$\cup_{i=1}^{\ell} g\,g_i\,L = \cup_{i=1}^{\ell} g_i\,L$ for all $g\in G$. In other words, each $g\in G$ induces a permutation of the subspaces $g_i\, L$. Finally, the measurable sub-bundles
	$\Lscr_i:=\{ (\theta, M(\theta)\, g_i\,v)\colon \theta\in\T^d, \, v\in L\} $ are permuted by the action of each element $(\beta, B)\in \Gscr_{\muh}$. This proves (2).
\end{proof}

%\begin{remark}
%The converse implication of (2)   also holds, provided we assume $(f, \muh^\Z\times m)$ to be  ergodic, see~\cite{CDK-notestrongirred}.
%\end{remark}

\medskip

\subsection*{A Furstenberg type criterion for positivity}

H. Furstenberg~\cite[Theorem 8.6]{Fur})  established a criterion for the positivity of the maximal Lyapunov exponent of random linear cocycles. Later this criterion was extended to a more abstract framework, see~\cite[Theorem 3.17]{FurmanSurvey} and~\cite[Theorem 6.1]{FurmanMonod}). The proof of the criterion in this setting requires that the group be locally compact, which is of course not the case for our group $\Gscr$ of quasiperiodic cocycles.

We formulate and prove this criterion in our mixed setting, combining arguments from the works of Furstenberg and Furman.
%, as well as from Avila, Viana and others, through the so called \textit{invariance principle}.

\begin{theorem}	\label{Furstenberg-Furman}
Given $\muh\in \Prob_c(\Gscr)$ such that
\begin{enumerate}
	\item  $f$ is  ergodic w.r.t. $\muh^\Z\times m$,
	\item $\muh$ is non-compact,
	\item $\muh$ is strongly irreducible,
\end{enumerate}
then $\fle(\muh)>0$.
\end{theorem}

\begin{proof}
Assume by contradiction that $\fle(\muh)=0$.	By Proposition~\ref{IP Lemma}, the measure $\eta$ is invariant under the group generated by $\supp(\muh)$. Hence, by Proposition~\ref{Furs Lemma Gscr}, $\muh$ is not strongly irreducible which contradicts assumption (3). Therefore $\fle(\muh)>0$.
\end{proof}

It remains to establish Propositions~\ref{Furs Lemma Gscr} and~\ref{IP Lemma}. We begin with the following lemma of H. Furstenberg.
\begin{lemma} \label{Furs Lemma}
	Given a subgroup $G\subseteq \SL_m(\R)$
	 and a probability measure $\pi\in\Prob(\Proj)$, assume
	\begin{enumerate}
		\item $G$ is non-compact,
		\item $\forall g\in G $,\; $g_\ast \pi=\pi$.
	\end{enumerate}
	Then $G$ is not strongly irreducible.
\end{lemma}

\begin{proof}
This lemma is a by-product of the  proof of~\cite[Theorem 8.6]{Fur}.
It follows also from~\cite[Corollary 3.2.2]{Zimmer}.
\end{proof}

We also need the following extension of the previous lemma to groups of cocycles due to A. Furman (see ~\cite[Lemma 5.2]{FurmanMonod}).

\begin{proposition} \label{Furs Lemma Gscr}
Given $\muh\in\Prob_c(\Gscr)$ and $\eta\in\Prob(\QP)$ which projects down to $m$, assume
\begin{enumerate}
	\item  $f$ is  ergodic w.r.t. $\muh^\Z\times m$,
	\item $\muh$ is non-compact (Definition~\ref{non-compact/strongly irreducible, Zariski: Gscr}),
	\item $(\beta,B)_\ast \eta=\eta$,\; $\forall (\beta,B)\in \supp(\muh) $.
\end{enumerate}
Then $\muh$ is not strongly irreducible.
\end{proposition}

\begin{proof}
The fact that the group $\Gscr$ is not locally compact (an assumption of~\cite[Lemma 5.2]{FurmanMonod}) does not affect the conclusion, as we will now see.

Consider the disintegration $\{\eta_\theta\}_{\theta\in\T^d}$ of the measure  $\eta$ w.r.t. the canonical projection $\pi\colon \QP=\T^d\times\Proj \to\T^d$. The invariance (3) means that
\begin{equation}
\label{desintegration invariance}
\eta_{\theta+ \beta} = B_\ast\eta_\theta\quad \forall (\beta, B)\in \Gscr_{\muh} \; \text{ and }  \text{ a.e. } \theta\in\T^d.
\end{equation}

Consider also the (push-forward) action of $\SL_m(\R)$ on $\Prob(\Proj)$. This action is tame \footnote{
	Let $G\times X\to X$ be a continuous action of a topological group $G$ acting on a second countable topological space $X$. This action is called {\em tame} if there exists a countable family of Borel sets $\{A_i\}_{i\in \N}$ in $X/G$ which separates points.
}
 by~\cite[Theorem 4.1]{FurmanMonod}.
By~\cite[Proposition 2.1.10]{Zimmer}, every quasi-invariant ergodic measure on $\Prob(\Proj)$ is supported on a single
 orbit of the action of $\SL_m(\R)$ on $\Prob(\Proj)$.
Note that in~\cite[Definition 2.1.9]{Zimmer}	{\em tame} actions are called {\em smooth}. Assume $G\times S\to S$ is a continuous action of a topological group $G$ on a compact metric space $X$.

\begin{definition}[See Section 2.1 of~\cite{Zimmer}]
 A measure $\nu\in\Prob(S)$ is called {\em quasi-invariant} under the action of $G$ if  for every Borel set $B\subseteq S$ and all $g\in G$,
$\nu(g^{-1} B)=0$ if and only if   $\nu(B)=0$.
A quasi-invariant measure $\nu\in\Prob(S)$ is called {\em ergodic} if for any $G$-invariant Borel set $B\subseteq S$, either $\nu(B)=0$
or  $\nu(B)=1$.
\end{definition}

Let $\lambda\in \Prob(\SL_m(\R))$ be the (left-invariant) Haar measure. Take a continuous density  $\rho:\SL_m(\R)\to (0,+\infty)$ with $\int \rho\, d\lambda=1$. The probability measure $\lambda_\rho:= \rho\, \lambda$ is equivalent to $\lambda$, in the sense that they share the same null sets.
Consider  the measure $\nu\in \Prob(\Prob(\Proj))$ defined by
\begin{equation}
 \nu := \int_{\T^d}
  \lambda_\rho\ast \delta_{\eta_\theta}\, d\theta .
\end{equation}
The measure $\lambda_\rho\ast \delta_{\eta_\theta}$ has support
$\SL_m(\R)\, \eta_\theta:= \{ g_\ast\eta_\theta\colon g\in\SL_m(\R) \} $. Hence
$ \supp(\nu)\supseteq  \SL_m(\R)\, \eta_\theta  $ for a.e. $\theta\in\T^d$.
We claim that $\nu$ is quasi-invariant and ergodic for the action of
$\SL_m(\R)$ on $\Prob(\Proj)$. Before justifying this claim let us finish the proof.

From the claim,  by~\cite[Proposition 2.1.10]{Zimmer} there exists a measure $\pi\in \Prob(\Proj)$ such that $\supp(\nu)\subseteq  \SL_m(\R)\, \pi$.
Hence $\eta_\theta\in \SL_m(\R)\, \pi$ for a.e. $\theta\in\T^d$.
By Kuratowski and Ryll-Nardzewski's measurable selection theorem~\cite[Theorem 5.2.1]{Srivastava} there exists
$\Phi:\T^d\to \SL_m(\R)$ Borel measurable such that
$\eta_\theta=\Phi(\theta)_\ast\pi$ for a.e. $\theta\in\T^d$.
To check that the hypothesis of this theorem holds we need to prove that for any open
set $U\subset \SL_m(\R)$ the set
$\Mscr_{\theta, U}:=\{\theta\in \T^d\colon \eta_\theta \in U\, \pi \}$ is Borel measurable. Without loss of generality we can assume that the map
$\T^d\ni \theta\mapsto \eta_\theta\in \Prob(\Proj)$ is Borel measurable.
On the other hand $U\,\pi$ is also Borel measurable because it can be written as a countable union of compact sets.  Hence 
$\Mscr_{\theta, U}=\eta^{-1}(U\,\pi)$ is a Borel set and the hypothesis holds.   Therefore the cocycle
$C:\Gscr_\mu\times \T^d\to\SL_m(\R)$,
$$ C((\beta, B), \theta):= \Phi(\theta+\beta)^{-1}\, B(\theta)\, \Phi(\theta)$$
takes values in the closed subgroup
$G:=\{ g\in\SL_m(\R) \colon g_\ast\pi=\pi \}$ of $\SL_m(\R)$.
By assumption (2)\, $G$ is non-compact,
while by construction $g_\ast \pi=\pi$ for all $g\in G$.
Thus, from Lemma~\ref{Furs Lemma} we conclude that $G$ is not strongly irreducible, i.e., $G$ is virtually reducible.
This implies that $\muh$ is not strongly irreducible.

Regarding the claim we prove now that
$\nu$ is quasi-invariant  for the action of
$\SL_m(\R)$ on $\Prob(\Proj)$. Let us say that two measures
$\mu_1$, $\mu_2$ on the same topological space are equivalent, and write $\mu_1\sim \mu_2$, when they share the same null sets (zero measure sets).
Then $\lambda\sim \lambda_\rho$ and for any $g\in\SL_m(\R)$,
$\lambda_\rho\sim \lambda=g_\ast \lambda\sim g_\ast \lambda_\rho$.
Next lemma implies that $\lambda_\rho\ast \delta_{\eta_\theta} \sim  g_\ast \lambda_\rho\ast \delta_{\eta_\theta}$ for all $\theta\in \T^d$ and $g\in\SL_m(\R)$.

\begin{lemma}
	Given $\mu_1,\mu_2\in\Prob(\SL_m(\R))$ and $\pi\in \Prob(\Proj)$,\\
	if $\mu_1\sim \mu_2$ \, then \, $\mu_1\ast \delta_\pi \sim  \mu_2\ast \delta_\pi$.
\end{lemma}

\begin{proof} Notice
	$(\mu_i\ast \delta_\pi)(B) = \mu_i\{g\in \SL_m(\R) \colon g_\ast\pi \in B\} $, for $i=1,2$,  and use that $\mu_1\sim \mu_2$.
\end{proof}

The quasi-invariance of $\nu$ corresponds to  $\nu \sim g_\ast \nu$
for all $g\in \SL_m(\R)$. This conclusion is straightforward from the next lemma.

\begin{lemma}
Given a compact metric space $X$ and  measures $\mu_i=\int \mu^\theta_i\, d\theta$, $i=,1,2$, in $\Prob(X)$, both obtained  averaging measurable functions $\T^d\ni \theta \mapsto \mu^\theta_i\in\Prob(X)$, if $\mu^\theta_1\sim \mu^\theta_2$ for a.e. $\theta$ then $\mu_1\sim \mu_2$.
\end{lemma}

\begin{proof}
Given a Borel set $B\subseteq X$, if
$\mu_1(B)=\int \mu^\theta_1(B)\, d\theta=0$ then $\mu_1^\theta(B)=0$ for a.e. $\theta$, and by the assumption  also  $\mu_2^\theta(B)=0$ for a.e. $\theta$.
Hence $\mu_2(B)=\int \mu^\theta_2(B)\, d\theta=0$. This proves that $\mu_2\ll \mu_1$. Exchanging the roles of $\mu_1$ and $\mu_2$ we get $\mu_1\sim\mu_2$.
\end{proof}

To finish, we  prove that
$\nu$ is ergodic.
Let $\B\subseteq \Prob(\Proj)$ be an $\SL_m(\R)$-invariant Borel set.
We want to see that $\nu(\B)\in \{0,1\}$. First notice that
$ \nu(\B)= \int_{\T^d} \varphi_\B(\theta)\, d\theta $ where
$\varphi_\B(\theta):= \lambda_\rho\{g\in\SL_m(\R)\colon g_\ast \eta_\theta\in \B \}$. Since $\B$ is $\SL_m(\R)$-invariant, the function $\varphi_\B=\ind_E$ matches the indicator function  of
$E:=\{\theta\in\T^d\colon \eta_\theta\in \B\}$.
For any  $(\beta, B)\in\Gscr_{\muh}$, using the invariance relation~\eqref{desintegration invariance} we have for m-a.e. $\theta$
\begin{align*}
\theta\in E \, \Leftrightarrow \,  \eta_\theta\in \B \,  \Leftrightarrow \,
B_\ast \eta_\theta \in \B \, \Leftrightarrow \,  \theta+\beta\in E .
\end{align*}
Hence $\ind_E$ is an $m$-stationary observable and by~\cite[Theorem 2.3 item (3)]{CDK-paper1}
it must be constant.
Therefore $\nu(\B)=m(E)$ is either $0$ or $1$.
\end{proof}

\begin{proposition} \label{IP Lemma}
Given $\muh\in\Prob_c(\Gscr)$, assume
\begin{enumerate}
	\item $(f, \muh^\Z\times m)$ is ergodic,
	\item $\fle(\muh)=0$.
\end{enumerate}
Then there exists  $\eta\in\Prob_{\muh}(\QP)$ such that  $(\beta,B)_\ast \eta=\eta$,\; $\forall (\beta,B)\in \Gscr_{\muh} $.
\end{proposition}

\begin{proof}
Let  $m$ be the Haar measure on $\T^d$, $\Pmeas$ be the Riemannian volume measure on $\Proj$. We denote the product measure by $\Qmeas:=m\times\Pmeas$. The three notations are intentionality confusable, since they all stand for uniform measures. Less harmless, the reader should mind  that $m$ also stands for the  dimension of the Euclidean fiber $\R^m$.

Given $g\in\SL_m(\R)$, let $\varphi_g\colon \Proj\to \Proj$
denote the projection action of the matrix $g$, $\varphi_g(\hat x):= \hat g\, \hat x$.

From now on, we make the convention that $(g^{-1})_\ast \Pmeas$ and $(\omega^{-1})_\ast\eta$ will be simplified to $g^{-1}\Pmeas$ and $\omega^{-1}\eta$ respectively ($\omega$ is an element of $\Gscr$).

\begin{lemma}
\label{Jacobian expression}
Given $g\in \SL_m(\R)$,
for every $\hat x\in\Proj$ and every unit vector $x\in\hat x$,
$$  \log  \norm{g\,x}   =  -\frac{1}{m}\, \log \det (D\varphi_g)_{\hat x}
= -\frac{1}{m}\, \log \frac{d g^{-1} \Pmeas }{d \Pmeas}(\hat x) $$	where $\frac{d g^{-1} \Pmeas }{d \Pmeas}$ is the density of the measure $g^{-1} \Pmeas$ w.r.t. $\Pmeas$.
\end{lemma}

\begin{proof} The proof is straightforward and we omit it.
\end{proof}

Recall that $\alfa:\Gscr\to\T^d$ and $\pi:\QP\to\T^d$ are the projections $\alfa(\beta, B):=\beta$ and $\pi(\theta, \hat p):=\theta$. Let $\mu:= \alfa_\ast\muh$ and $\{\muh_\beta\}_\beta$
be the disintegration of $\muh$ w.r.t. $\alfa$, i.e.,
$\muh = \int \muh_\beta\, d\mu(\beta)$.

Given $\theta\in\T^d$, the projection  $\Mscr_\theta\colon \Gscr\to \SL_m(\R)$,  $\Mscr_\theta(\beta, B):=B(\theta)$,
determines the measure $\muh(\theta) := (\Mscr_\theta)_\ast\muh \in\Prob(\SL_m(\R))$. Similarly,  $\muh_\beta(\theta):=    (\Mscr_\theta)_\ast\muh_\beta$.

\begin{definition}
\label{class Cinfty def}
We say that $\muh\in \Prob_c(\Gscr)$ is of class $B_\infty$
if there exists $h:\T^d\times \T^d\times\SL_m(\R)\to \R$ bounded, measurable, non-negative and compactly supported such that $\muh_\beta(\theta) =h(\beta, \theta, \cdot)\, \lambda$ for $\mu$-a.e. $\beta$ and  $m$-a.e. $\theta$, where
 $\lambda$ stands for the Haar measure on $\SL_m(\R)$.
\end{definition}

Let us say a few words about approximating arbitrary measures in $\Prob_c(\Gscr)$ by $B_\infty$ measures.
Take an arbitrary bounded and measurable compactly supported  function $\rho\colon \SL_m(\R)\to [0,+\infty)$  and define $\lambda_\rho=\rho\,\lambda$.
Since $\SL_m(\R)$ embeds homomorphically in $\Gscr$, where each matrix $g$ is associated with the constant cocycle $(0,g)$, this embedding induces an action of $\SL_m(\R)\times \Gscr \to \Gscr$
and (through it) a convolution action of the semigroup $\Prob_c(\SL_m(\R))$ on $\Prob_c(\Gscr)$.

\begin{lemma}\label{lemdis}
For any $\muh\in \Prob_c(\Gscr)$, $\lambda_\rho\ast \muh$ is of class $B_\infty$.  Moreover, $(\lambda_\rho\ast \muh)_{\beta}(\theta) = \lambda_\rho\ast \muh_{\beta}(\theta)$ \, for $\mu$-a.e. $\beta$ and $m$-a.e. $\theta$.
\end{lemma}

\begin{proof}
	 Straightforward.
 \end{proof}

If we now take $\{\rho_i\}_{i\in\N}$ to be an approximation of the identity, we have $\lambda_{\rho_i}\to \delta_I$
(Dirac measure sitting at the identity matrix) in the weak-$\ast$ topology as $i\to \infty$, which in turn implies that
$\lambda_{\rho_i}\ast \muh \to \muh$ in  the weak-$\ast$ topology as $i\to \infty$.

In the following, we recall some elementary results in geometry.

Assume that $M$ and $N$ are Riemmanian manifolds of the same dimension and $f: M\to N$ is a local diffeomorphism. Denote by $m_M$ and $m_N$ the Riemmanian volume on $M$ and $N$ respectively. Then we have
$
f_\ast m_M=h\, m_N
$
where
$$
h(y)=\sum_{x\in f^{-1}(y)}\frac{1}{Jf_x}.
$$
Here $Jf_x=\lvert \det Df_x\rvert$ is the Jacobi determinant. More generally, given $\rho: M \to [0,+\infty)$ we have
$
f_\ast(\rho\, m_M)=(\mathcal{L}_f\rho)m_N
$
where $h=\mathcal{L}_f\rho$ is given by
$$
h(y)=\sum_{x\in f^{-1}(y)}\frac{\rho(x)}{Jf_x}.
$$

Assume now $f: M \to N$ is onto and its differential $Df_x: T_xM\to T_{f(x)}N$ is also onto, i.e. $f$ is a submersion. Then $f^{-1}(y)\subset M$ is a smooth manifold of dimension $k=\dim M-\dim N$. Let $\mu_y$ be the Riemannian volume on $f^{-1}(y)$, then $f_\ast(\rho\,m_M)=h\,m_N$ where
$$
h(y)=\int_{f^{-1}(y)}\frac{\rho(x)}{Jf_x}d\mu_y(x).
$$
Here $Jf_x=\prod\limits_{j=1}^{k}s_j(Df_x)>0$ for any $x\in M$.

From the previous basic facts, we immediately obtain:

\begin{lemma}\label{subm}
Assume $f: M \to N$ is a submersion, if $\mu \ll m_M$ then $f_\ast\mu\ll m_N$. Moreover if $\mu \sim m_M$ then $f_\ast\mu\sim m_N$.
\end{lemma}

Fix $p\in \Proj$ with $\norm{p}=1$, define $F_{p}:\SL_m(\R) \to \Proj$ by $F_{p}(g)=\hat g \hat p$.
This map is a submersion where $(DF_{p})_g: T_g\SL_m(\R) \to T_{\hat g \hat p}\Proj$,
$$ (DF_{p})_g(h) = \frac{h p}{\norm{g p}} - \langle \frac{hp}{\norm{g p}}, \frac{gp}{\norm{g p}}\rangle\, \frac{gp}{\norm{g p}}
= \left[ \frac{h p}{\norm{g p}} \right]^T  $$
and $[v]^T$ denotes the component of $v$ tangent to $T_p\Su^{m-1}\equiv T_{\hat p} \Proj$.
In fact, note that $F_p$ is equivariant  under the left action of $\SL_m(\R)$ on itself, i.e. $F_p(gh)=gF_p(h)$. Therefore, it is sufficient to see that $(DF_p)_I: T_I\SL_m(\R) \to T_{\hat p}(\Proj)$ is onto. However, given $v\in T_{\hat p}(\Proj)$ with $\norm{v}=1$, there exists $h$ with $\tr(h)=0$ such that $h p=v$. This proves that $F_p$ is a submersion.

\begin{remark}\label{remabs}
By a simple calculation, $\lambda_\rho \ast \delta_p= (F_p)_\ast \lambda_\rho$. Thus by Lemma \ref{subm}, we have $\lambda_\rho \ast \delta_p= (F_p)_\ast \lambda_\rho \ll \Pmeas$, which implies that $\lambda_\rho \ast \zeta= \int \lambda_\rho \ast \delta_p d\zeta(p) \ll \Pmeas$ for any $\zeta \in \Prob(\Proj)$.
\end{remark}

In the following, we prove the equivalence of the stationary measure $\eta\in \Prob_{\muh}(\QP)$ and the product measure $\Qmeas$ under certain conditions.

\begin{proposition}\label{equim}
If $\muh\in\Prob_c(\Gscr)$ is  of class $B_\infty$ then for any $\eta\in\Prob_{\muh}(\QP)$ satisfies $\eta \ll \Qmeas$. Moreover, If $\supp(\rho)$ contains an $\SO_m(\R)$ left invariant open set and if $\eta$ is stationary w.r.t. $\muh=\lambda_\rho\ast \tilde{\nu}$ for some $\tilde{\nu}\in \Prob_c(\Gscr)$, then $\eta \sim \Qmeas$.
\end{proposition}

\begin{proof}
From Proposition~\ref{invariance disintegration} we have
\begin{align*}
\eta_\theta &= \int_\Gscr  B(\theta-\beta)_\ast \eta_{\theta-\beta} \, d\muh(\beta, B) \\
&=  \int_{\T^d} \int_{C^0(\T^d,\SL_m(\R))} B(\theta-\beta)_\ast \eta_{\theta-\beta} \, d\muh_\beta(B) \, d\mu(\beta) \\
&=  \int_{\T^d} \int_{\SL_m(\R)} g_\ast \eta_{\theta-\beta} \; d\muh_\beta(\theta-\beta)(g) \, d\mu(\beta)\\
&=\int_{\T^d}  \muh_\beta(\theta-\beta)\ast \eta_{\theta-\beta}\, d\mu(\beta).
\end{align*}
Since $\muh_\beta(\theta-\beta)\ll \lambda$ for $\mu$-a.e. $\beta$ and $m$-a.e. $\theta$, by the same argument as Remark \ref{remabs} we get that $\eta_\theta\ll \Pmeas$ for $m$-a.e. $\theta$. This implies that $\eta \ll \Qmeas$.

On the other hand, if $\muh=\lambda_\rho \ast \tilde{\nu}$, then by Lemma \ref{lemdis} we have $\muh_\beta(\theta)=\lambda_\rho \ast \tilde{\nu}_\beta(\theta)$ for $\mu$-a.e. $\beta$ and $m$-a.e. $\theta$. This shows that $\muh_\beta(\theta-\beta)\ast \eta_{\theta-\beta}=\lambda_\rho \ast \zeta$ (which is equivalent to $\Pmeas$ for $\mu$-a.e. $\beta$ and $m$-a.e. $\theta$ because $\supp(\rho)$ contains an $\SO_m(\R)$ left invariant open set) for some $\zeta\in \Prob(\Proj)$. Therefore $\eta_\theta \sim \Pmeas$ for $m$-a.e. $\theta$ and $\eta\sim \Qmeas$.
\end{proof}

\bigskip

\noindent
{\bf First Case}: We assume that $\muh\in\Prob_c(\Gscr)$ is of class $B_\infty$ with a stationary measure $\eta\in \Prob_{\muh}(\QP)$ that is equivalent to $\Qmeas$. Under the same assumptions we prove that $(\beta,B)_\ast \eta=\eta$,\; $\forall (\beta,B)\in \Gscr_{\muh} $.

\begin{proposition}
	For any  $\eta\in\Prob_\muh(\QP)$    equivalent to $\Qmeas$,
	$$\alpha_\muh(\eta) =  -\frac{1}{m}\,
	\int_{\Gscr} \left[ \int_{\QP} \log \frac{d \omega^{-1} \eta }{d \eta}  \, d\eta \right] \,d\muh(\omega)\; .$$
\end{proposition}

\begin{proof} A cocycle $\omega=(\beta, B)\in\Gscr$ induces
	the map  $\varphi:\QP\to \QP$, defined by
	$\varphi_\omega(\theta, \hat p):= (\theta+\beta, \hat B(\theta)\,\hat p)$, which by Lemma~\ref{Jacobian expression} satisfies
	$$  - \frac{1}{m}\, \log \det D \varphi_\omega (\theta, \hat p)
	=\log \norm{B(\theta)\, p} . $$
	Hence
	\begin{align*}
	\alpha_\muh(\eta) &= \int_{\Gscr} \int_{\QP} \log \norm{B(\theta)\,p}\, d\eta(\theta,\hat p) \,d\muh(\beta, B) \\
	&= -\frac{1}{m}\,\int_{\Gscr} \left[ \int_{\QP}  \log \frac{d \omega^{-1} \Qmeas }{d \Qmeas} \, d\eta \right] \,  d\muh(\omega ) .
	\end{align*}
	Call $\gamma$ to the right hand side in the Proposition statement. Then

	\begin{align*}
 \alpha_\muh(\eta)-\gamma &= \int_{\Gscr} \int_{\QP}  \log\left(
	\frac{d \omega^{-1} \Qmeas}{d\Qmeas}\left/\frac{d \omega^{-1} \eta}{d\eta}\right. \right)\,d\eta\,d\muh(\omega) \\
	&= \int_{\Gscr} \int_{\QP} \log\left(
	\frac{d \omega^{-1} \Qmeas}{d \omega^{-1} \eta}\left/\frac{d\Qmeas}{d\eta}\right. \right)\,d\eta \,d\muh(\omega) \\
	&= \int_{\Gscr} \int_{\QP} \left[ \log
	\frac{d \omega^{-1} \Qmeas}{d \omega^{-1} \eta}- \log \frac{d\Qmeas}{d\eta}  \right] \,d\eta\,d\muh(\omega) \\
	&= \int_{\Gscr} \left[ \int_{\QP} \log
	\frac{d \omega^{-1} \Qmeas}{d \omega^{-1} \eta}  \,d\eta \right]  \,d\muh(\omega)
	- \int_{\QP} \log \frac{d\Qmeas}{d\eta} \,d\eta \\
	&= \int_{\Gscr} \left[ \int_{\QP} \log
	\left( \frac{d \Qmeas}{d \eta}\circ \varphi_\omega \right)  \,d\eta \right] \,d\muh(\omega)
	- \int_{\QP} \log \frac{d\Qmeas}{d\eta} \,d\eta  \\
	&=   \int_{\QP} \log
	\frac{d \Qmeas}{d \eta}  \,d(\muh\ast\eta)
	- \int_{\QP} \log \frac{d\Qmeas}{d\eta} \,d\eta  =0
	\end{align*}
	where on the last step we use that  $\muh\ast\eta  = \eta$.
\end{proof}

Assume $\fle(\muh)=0$, then $\alpha_\muh(\eta)=0$ for any $\eta\in \Prob_{\muh}(\QP)$ due to Theorem \ref{NonRandomFiltrarion}. By Jensen's inequality,
%for any $g\in \SL_d(\R)$,
\begin{align*}
0= \int_\Gscr  \left[ \int_\QP \log \frac{d \omega^{-1}\eta}{d\eta} \, d \eta \right] \, d\muh(\omega) &\leq
\log  \int_\Gscr \left[ \int_\QP \frac{d \omega^{-1}\eta}{d\eta} \, d \eta \right]\, d\muh(\omega) \\
&= \log 1 = 0 .
\end{align*}
Hence $d \omega^{-1} \eta/d\eta$ is constant  $\eta$-almost surely,
for $\muh$-a.e. $\omega$. Since it is the density of a probability measure it must be constant equal to $1$. This proves that
$\omega^{-1} \eta  = \eta$ for all $\omega\in\Gscr_\muh$, i.e., $\eta$ is $\Gscr_\muh$-invariant.

\bigskip
	
\noindent
{\bf General Case}: Assume now  $\muh\in\Prob_c(\Gscr)$ is  arbitrary. Consider an approximation of the identity $\{\rho_n\colon \SL_m(\R)\to \R \}_{n\in\N}$ consisting of continuous, compactly supported and non-negative function with an $\SO_m(\R)$ left invariant open set in the support.
Then for each $n\geq 1$, $\muh_n:= \lambda_{\rho_n}\ast \muh$ is of class $B_\infty$ by Lemma \ref{lemdis} and has a stationary measure $\eta_n\in \Prob_{\muh_n}(\QP)$ equivalent to $\Qmeas$ by Proposition \ref{equim}.
Possibly taking a subsequence, we can assume that
$\eta_n$ converges in the weak-$\ast$ topology to some
$\eta\in\Prob(\Q)$.  Then by continuity of the convolution operation  we have $\muh\ast\eta=\eta$.

\begin{align}
0 &= \alpha_\muh(\eta) =  \int_{\Gscr} \int_{\QP} \log \norm{B(\theta)\,p}\, d\eta(\theta, \hat p)\,d\muh(\beta, B)  \nonumber \\
& = \lim_{n\to\infty} \int_{\Gscr}\int_{\QP} \log \norm{B(\theta)\,p}\, d\eta_n(\theta, \hat p)\,d\muh_n(\beta, B)
\nonumber \\
&= \lim_{n\to\infty} -\frac{1}{m}\,
\int_{\Gscr} \left[ \int_{\QP} \log \frac{d \omega^{-1}\eta_n }{d \eta_n}  \, d\eta_n \right]\,d\muh_n(\omega)\;.
\label{alpha:mu}
\end{align}

In the next lemma we consider the total variation norm of a signed measure $\eta$
$$\norm{\eta}:=\sup_{\vert \phi\vert \leq 1 } \abs{\int \phi d\eta } $$
where the sup is taken over all continuous functions.
Notice the action of every $\omega\in \Gscr$  induces a diffeomorphism $\varphi_\omega:\QP\to \QP$. Hence $\norm{\omega\, \eta}=\norm{(\varphi_\omega)_\ast\eta }= \norm{\eta}$.

\begin{lemma}
	\label{||pi1-pi2|| leq }
	Given $\eta_1,\eta_2\in\Prob(\QP)$, if $\eta_1$ is equivalent to $\eta_2$,
	$$ \norm{\eta_1-\eta_2}^2\leq -4\, \int_{\QP} \log \frac{d\eta_2}{d\eta_1} \,d\eta_1  . $$
\end{lemma}

\begin{proof}
Same as in~\cite[Lemma 8.10]{Fur}.
\end{proof}

Thus, from~(\ref{alpha:mu}) and  Lemma~\ref{||pi1-pi2|| leq },
we get that
$$ \lim_{n\to\infty}
\int_{\Gscr} \norm{ \omega^{-1}  \,\eta_n-\eta_n}^2\, d\muh_n(\omega) = 0. $$

Take a density point $\omega$ of $\muh$, $\delta>0$ and a neighbourhood $U$  of $\omega$ in $\Gscr$.
Because  $\muh$ is the weak-$\ast$ limit of $\muh_n$,  there exist $\varepsilon>0$ and $n_0\in\N$ such that $\mu_n(U)> \varepsilon$ and for all $n\geq n_0$,
$$  a_n:=\int_{\Gscr} \norm{ \omega^{-1} \,\eta_n-\eta_n}^2\, d\muh_n(\omega) <\delta^2\,\varepsilon .$$

Assume, by contradiction, that for every $\omega_1\in U$, $\norm{\omega_1^{-1}\,\eta_n-\eta_n}\geq \delta$.
Integrating in $\omega_1\in U$ we  have for all $n\geq n_0$
$$ a_n \geq \int_{U} \norm{ \omega_1^{-1} \eta_n-\eta_n}^2\, d\muh_n(\omega_1)
\geq \delta^2\, \varepsilon > a_n . $$
Thus there exists $\omega_1\in U$ such that
$\norm{\omega_1^{-1}\,\eta_n-\eta_n}< \delta$, and since $\delta$ and $U$ can be made arbitrary small
there exists a sequence $\omega_n$ converging to $\omega$ such that   $\norm{\eta_n-\omega_n\,\eta_n} = \norm{\omega_n^{-1}\,\eta_n-\eta_n}\to 0$.

Hence, for any density point $\omega$  of $\muh$, we have
$\omega\,\eta_n-\eta_n \to  0$ in the weak-$\ast$ topology.
In fact given  $\phi\in C^0(\QP)$ and setting  $L_\omega(\phi)(\theta,\hat p):=\phi(\omega\cdot( \theta,\hat p))$ we have
for any density point $\omega$  of $\muh$,
\begin{align*}
\abs{ \smallint \phi\, d(\omega\,\eta_n)- \smallint \phi\, d\eta_n} &\leq
\abs{ \smallint \phi\, d(\omega\,\eta_n)- \smallint \phi\, d(\omega_n\,\eta_n)}  \\
&\qquad \qquad +  \abs{ \smallint \phi\, d(\omega_n\,\eta_n)-\smallint \phi\, d\eta_n} \\
&\leq \abs{ \smallint L_\omega(\phi)\, d\eta_n- \smallint L_{\omega_n}(\phi)\, d\eta_n}  + \norm{ \omega_n\,\eta_n-\eta_n} \\
&\leq  \underbrace{ \norm{  L_\omega(\phi) -  L_{\omega_n}(\phi) }_\infty   }_{\to 0} + \underbrace{ \norm{ \omega_n\,\eta_n-\eta_n}  }_{\to 0}\to 0 .
\end{align*}
This convergence implies that $\omega\,\eta=\eta$.
Since the $\muh$-density points are dense in $\supp(\muh)$ it follows that $\eta$ is $\Gscr_\muh$-invariant.	
\end{proof}

\bigskip

\subsection*{Continuity of the first Lyapunov exponent }

Quasi-irreducible measures in $\Prob_c(\Gscr)$, introduced in
Definition~\ref{irred & quasi-irred} (recall also their characterization in Proposition~\ref{quasi-irred charact}), constitute a generic class of measures. 
Next proposition establishes the generic continuity of the first Lyapunov exponent, i.e. in the class of quasi-irreducible measures.
For random linear cocycles this result was established in H. Furstenberg and Y. Kifer~\cite[Proposition 4.1]{FKi}. Later Y. Kifer~\cite[Theorem 2.2]{Kifer-book} extended this result to a much broader setting that includes ours.

\begin{proposition}[Kifer] 	\label{Kifer continuity}
Given any compact set $\mathscr{K}\subset\Gscr$, the  Lyapunov exponent, $\muh\mapsto \fle(\muh)$,  is a continuous function on the space  of quasi-irreducible probability measures $\muh\in\Prob_c(\Gscr)$
with $\supp(\muh)\subset \mathscr{K}$.
\end{proposition}

\begin{proof}
Assume $\muh_n\to \muh$ is the weak-$\ast$ topology,
with $\supp(\muh_n)\subset \mathscr{K}$ for all $n\geq 1$.
Take $L<\infty$ such that $\norm{B}\leq L$ for all
 $(\beta, B)\in \muh_n$.
For every $n\geq 1$, choose a maximizing $\muh_n$-stationary measure
$\eta_n\in\Prob(\QP)$, so that $\fle(\muh_n)=\alpha_{\muh_n}(\eta_n)$.
Take any sublimit $\eta$ of the sequence $\eta_n$, then $\eta$ is $\muh$-stationary.

For each $\omega=(\beta, B)\in \Gscr$, consider the function $\Psi_\omega : \QP\to \R$ in~\eqref{def Psi} and notice that $\norm{\Psi_\omega}_\infty \leq \log L$
for all $\omega\in \mathscr{K}$. Hence by the Dominated Convergence Theorem,
\begin{align*}
 \fle(\muh_n) &= \alpha_{\muh_n}(\eta_n) = \int_{\mathscr{K}} \left[ \int_{\QP} \Psi_\omega \, d\eta_n \right]\, d\muh_n(\omega)\\
 &\qquad \qquad \longrightarrow
\int_{\mathscr{K}} \left[ \int_{\QP} \Psi_\omega \, d\eta \right]\, d\muh(\omega) = \alpha_{\muh}(\eta)  =\fle(\muh) ,
\end{align*}
where the last equality holds because $\muh$ is a quasi-irreducible measure (see Proposition~\ref{quasi-irred charact}).
The fact that the limit is the same for all sublimits $\eta$, implies that $\fle(\muh_n)$ converges to $\fle(\mu)$.
\end{proof}

\subsection*{Sufficient conditions for the Furstenberg type criterion}

We present some more easily verifiable conditions that ensure the applicability of our criterion for the positivity of the Lyapunov exponent.

\begin{definition}[Monodromy]
The {\em monodromy group} of $\muh$ at $\theta$ is the group  $G_\theta(\muh)$  generated by the matrices $D(\theta)$, where $$(0,D) =(\beta_k, B_k) \cdot \ldots \cdot (\beta_1, B_1)$$
and for each $j=1,\ldots, k$, either
$(\beta_j, B_j)\in \supp(\muh)$ or else $(\beta_j, B_j)^{-1}=(-\beta_j, B_j^{-1})\in \supp(\muh)$.
\end{definition}

\begin{remark}
The monodromy group $G_\theta(\muh)$
contains the matrices $B_1^{-1}(\theta)\, B_2(\theta)$ and $B_1(\theta-\beta)\, B_2(\theta-\beta)^{-1}$ such that both $(\beta, B_1)$ and $ (\beta, B_2)$ belong to  $\supp(\muh)$  for some $\beta\in\T^d$.
\end{remark}

Given  a topological group  $G$, denote by $\Kscr(G)$ denote the set of all compact subsets
$K\subseteq G$ endowed with the usual Hausdorff distance
$$ d_H(K,K')=\inf \{ \delta>0 \colon K\subseteq B_\delta(K') \, \text{ and }\,  K'\subseteq B_\delta(K) \,
\} .$$
Consider the space $\Sscr(G)$  of closed subgroups of $G$ endowed with the {\em lower Chabauty topology}, defined as follows.
Given a closed subgroup $H\subset G$, a fundamental system of neighborhoods of $H$  for this topology consists of the sets
\begin{align*}
\Vscr_{K,\delta}(H) &:=\left\{ H'\in\Sscr(G)\colon \exists K'\in \Kscr(G)  \, \text{ such that }\,  K'\subseteq H' \right.\\
&\qquad\qquad\qquad\qquad\qquad\qquad\qquad \left. \, \text{ and } \, d_H(K,K')<\delta  \, \right\}
\end{align*}  
 where $\delta>0$ and $K\subseteq H$ is any compact subset.

\begin{definition}
A subgroup  $G\subset \SL_m(\R)$ is called {\em stably non-compact}, respectively {\em stably strongly irreducible}, {\em stably Zariski dense}, if it is an interior point of the set of all non-compact subgroups, respectively strongly irreducible, Zariski dense subgroups, w.r.t. lower Chabauty 
topology on $\Sscr(\SL_m(\R))$.
\end{definition}

\begin{proposition}	\label{Zariski criterion}
Given $\muh\in\Prob_c(\Gscr)$, if for some $\theta\in\T^d$ the mo\-nodromy group $G_\theta(\muh)$ is respectively stably non-compact, irreducible, stably strongly irreducible, stably Zariski dense, then $\Gscr_\muh$ and $\muh$
are  non-compact, irreducible, strongly irreducible, Zariski dense, respectively.	
\end{proposition}

\begin{proof}
Let $\mathscr{P}$ represent one of the following open properties:  stably non-compact, irreducible, stably strongly irreducible or  stably Zariski dense, and assume by contradiction that $\Gscr_\muh$ is respectively not `non-compact', reducible, virtually reducible or contained in a proper algebraic subgroup. See Definition~\ref{non-compact/strongly irreducible, Zariski: Gscr}.
Then  there exists $M\colon \T^d\to \SL_m(\R)$ measurable such that the cocycle $C:\Gscr_{\muh}\times\T^d\to \SL_m(\R)$ defined by
$$  C((\beta, B), \theta):= M(\theta+\beta)^{-1}\, B(\theta)
\,M(\theta) , $$
takes values in a proper closed subgroup  $H\subseteq \SL_m(\R)$
that is respectively compact, reducible, virtually reducible, or not Zariski dense.

From the definition above it follows immediately that
$$M(\theta)^{-1} G_\theta(\muh)\, M(\theta)\subseteq H \quad \text{ for }  \text{ a.e. } \theta\in\T^d.$$
Hence the group $G_\theta(\muh)$ does not satisfy $\mathscr{P}$
for a.e. $\theta\in\T^d$, but we still have to prove that $G_\theta(\muh)$ can not satisfy $\mathscr{P}$ for any $\theta$.
Assume by contradiction that there exists $\theta_0\in\T^d$ such that $G_{\theta_0}(\muh)$ satisfies $\mathscr{P}$.
Because $\Gscr$ is a group of continuous cocycles, the generators
of $G_{\theta}(\muh)$ depend continuously on $\theta$.
Property $\mathscr{P}$ is open in $\Sscr(\SL_m(\R))$ w.r.t. the lower Chabauty topology. 
%see~\cite{CDK-notestrongirred}. 
Therefore  $G_\theta(\muh)$ satisfies $\mathscr{P}$  for all $\theta$ in a neighborhood of $\theta_0$, which contradicts the previous conclusion.
\end{proof}

\begin{remark}
\label{SL2 stably strongly irreducible}
In $\SL_2(\R)$ non-compact and strongly irreducible cocycles (or subgroups) are always stably strongly irreducible.
%see~\cite{CDK-notestrongirred}.
\end{remark}

\begin{remark}
For measures in $\Prob_c(\Gscr)$, and $\Prob(\SL_m(\R))$, the following relations hold:
$$\text{ Zariski dense } \Rightarrow \left\{
\begin{array}{l}
\text{Strongly irred. } \Rightarrow \text{ Irreducible } \Rightarrow \text{ Quasi-irred. } \\
\text{Non-compact }
\end{array}
\right.$$
\end{remark}

\begin{remark}
Pairs of matrices $(A,B)\in\SL_m(\R)\times\SL_m(\R)$ generating a Zariski dense subgroup of $\SL_m(\R)$ are generic in both Haar measure sense and topological sense  (open and dense). See~\cite[Theorem 2.4]{BG2003}.
Therefore, Zariski dense measures are generic in $\Prob_c(\SL_m(\R))$. Hence by Proposition~\ref{Zariski criterion} it follows that Zariski dense measures are also generic in $\Prob_c(\Gscr)$. In particular the other classes  referred to in the previous remark are  generic too.
\end{remark}

\bigskip 

\section{Applications to Schr\"odinger Cocycles}\label{example}

We apply the general positivity and continuity results derived in the previous section to randomly perturbed quasiperiodic Schr\"odinger cocycles and obtain the positivity and the continuity of their Lyapunov exponents under mild assumptions.

\begin{example}
	\label{example 1}
	Given $\rho \in \Prob_c (\R)$, let
	\begin{equation*}
	\label{SL2 example}
	\muh_E := \delta_\alpha \times \int_\R \delta_{P(\omega) S_E}\,  d\rho(\omega)
	\end{equation*}
	be the family of mixed Schr\"odinger cocycles with fixed frequency $\alpha$ and random noise $\omega$ chosen with probability $\rho$, corresponding to the Schr\"odinger operator~\eqref{op1}. 
\end{example}

\begin{proposition}
	\label{ex 1}
	Consider the family of Schr\"odinger cocycles  defined by the measures
	$\muh_E$ in the previous example. If $\supp(\rho)$ has more than one element  then $\muh_E$ is non-compact and strongly irreducible. Thus
	$\fle(\muh_E)>0 $ $ \forall E\in\R$ and the map $E \mapsto \fle (\muh_E)$ is continuous.
\end{proposition}

\begin{proof}
	Assume $\omega_1\neq \omega_2$ are both in $\supp(\rho)$.
	For any $\theta$, the monodromy group $G_{\theta}(\muh_E)$ contains the matrices
	$$P(\omega_1)\, P(\omega_2)^{-1}=P(\omega_1-\omega_2) =\begin{bmatrix}
	1 & \omega_1-\omega_2\\ 0 & 1
	\end{bmatrix}$$
	and
	$$S_E (\theta)^{-1}\, P(\omega_2 -\omega_1)\, S_E (\theta) =\begin{bmatrix}
	1 & 0\\ \omega_1-\omega_2 & 1
	\end{bmatrix} .$$
	Since $\omega_1-\omega_2\neq 0$ and the product of these two matrices has trace $2+(\omega_1-\omega_2)^2$, the monodromy group is stably non-compact.  It is also not virtually reducible so by Remark~\ref{SL2 stably strongly irreducible} it is stably strongly irreducible.
	By Proposition~\ref{Zariski criterion}
	 $\muh_E$ is non-compact and strongly irreducible.
	By Theorem~\ref{Furstenberg-Furman} it follows that $\fle(\muh_E)>0$.
	Finally by Proposition~\ref{Kifer continuity}, this is a continuity point of $\fle$.	
\end{proof}

In the previous example we perturb a given quasiperiodic cocycle
by adding random real numbers.  We may consider more general  perturbations by adding random potentials.

\begin{example}
	\label{example 3}
Given
$\rho\in \Prob_c(C^0(\T^d,\R))$, let
\begin{equation}
 \muh_E:= \delta_\alpha \times \int_\R \delta_{P(\omega) S_E}\,  d\rho(\omega).
\end{equation}
Define $Z_{\omega_1,\omega_2}:=\{\theta\in \T^d : \omega_1(\theta)=\omega_2(\theta)\}$ for  any $\omega_1, \omega_2\in C^0(\T^d, \R)$.
\end{example}

\begin{proposition}\label{ex prop2}
Consider the family of Schr\"odinger cocycles defined by the measures
$\muh_E$. If $\supp(\rho)$ contains two elements $\omega_1$ and $\omega_2$ such that $Z_{\omega_1,\omega_2}\cap \left(Z_{\omega_1,\omega_2}-\alpha \right)=\emptyset$, then $\muh_E$ is non-compact and strongly irreducible. Thus 
$\fle(\muh_E)>0$, for all $E\in\R$ and 
these measures are continuity points of the first Lyapunov exponent.
\end{proposition}

\begin{proof}
Assume $\omega_1,\omega_2$ are both in $\supp(\rho)$.
For any $\theta$, the monodromy group $G_{\theta}(\muh_E)$ contains the matrices
\begin{align*}
P(\omega_1(\theta))\, P(\omega_2(\theta))^{-1} =P(\omega_1(\theta)-\omega_2(\theta))  =\begin{bmatrix}
 1 & \omega_1(\theta)-\omega_2(\theta)\\ 0 & 1
\end{bmatrix}
\end{align*}
and
\begin{align*}
&S_E (\theta-\alpha)^{-1}\, P(\omega_2(\theta-\alpha) -\omega_1(\theta-\alpha))\, S_E (\theta-\alpha)\\
=&\begin{bmatrix}
1 & 0\\ \omega_1(\theta-\alpha)-\omega_2(\theta-\alpha) & 1
\end{bmatrix} .
\end{align*}

Denote $r_i(\theta)=\omega_i(\theta)-\omega_i(\theta-\alpha), i=1,2$. Note that $\int r_i dm=0$ for $i=1,2$. Then there exists some $\theta_0$ such that $r_1(\theta_0)=r_2(\theta_0)$, i.e. $\omega_1(\theta_0)-\omega_2(\theta_0)=\omega_1(\theta_0-\alpha)-\omega_2(\theta_0-\alpha)$. By the assumption, $\omega_1(\theta_0)-\omega_2(\theta_0)$ and $\omega_1(\theta_0-\alpha)-\omega_2(\theta_0-\alpha)$ cannot be simultaneously zero. We conclude that they are both non-zero.

As in Proposition~\ref{ex 1}, the monodromy group $\Gscr_{\theta_0}(\muh_E)$ is stably non-compact and strongly irreducible. Hence by Proposition~\ref{Zariski criterion},
  $\muh_E$ is non-compact and strongly irreducible and by Theorem~\ref{Furstenberg-Furman} it follows that $\fle(\muh_E)>0$.
Finally by Proposition~\ref{Kifer continuity}, this is a continuity point of $\fle$.	
\end{proof}

In the following example we consider a Schr\"odinger cocycle
with randomly chosen frequencies.

\begin{example}
	\label{example 2}
	Given $\mu\in\Prob(\T^d)$, let
	\begin{equation*}
	\label{SL2 example 2}
	\muh_E:= \mu \times \delta_{S_E}
	\end{equation*}
	be the mixed random-quasiperiodic Schr\"odinger cocycle with fixed potential $v (\theta)$ but random frequency  $\alpha$ chosen with probability $\mu$, corresponding to the Schr\"odinger operator~\eqref{op2}. 
\end{example}

\begin{proposition}
	Consider the family of Schr\"odinger cocycles  driven by the measures $\muh_E$ above and determined by a probability $\mu \in \Prob (\T^d)$ and a potential function $v \colon\T^d\to\R$.
	If
	\begin{enumerate}
	\item  $\beta-\alpha$ is an ergodic translation of $\T^d$
		for some $\alpha,\beta\in \supp(\mu)$,
\item $v$ is analytic and  non-constant\footnote{Transversality between $v (\theta)$ and $v (\theta + \beta - \alpha)$ instead of analyticity is enough.},
\end{enumerate}
	then the measure $\muh_E$ is ergodic, non-compact and strongly irreducible. Thus
	$\fle(\muh_E)>0 $ $ \forall E\in\R$ and the map $E \mapsto \fle(\muh_E)$ is continuous.
	\end{proposition}

\begin{proof}
	Let $A=S_E$ and consider the  cocycles
	\begin{align*}
	(0, D_1)&:= (\beta, A)\cdot(\alpha, A)^{-1}\cdot(\beta, A)^{-1}\cdot (\alpha, A)\\
	(0, D_2)&:= (\beta, A)\cdot(\alpha, A)\cdot(\beta, A)^{-1}\cdot (\alpha, A)^{-1} .
	\end{align*}
	A simple calculation gives
	$$D_1(\theta) = A(\theta+\alpha-\beta)^{-1}\, A(\theta)
	=\begin{bmatrix}
	1 & 0 \\
	\varphi(\theta) & 1
	\end{bmatrix}$$
	with $\varphi(\theta):=V(\theta+\alpha-\beta)-V(\theta)$.
	Similarly,
	$$D_2(\theta) = A(\theta -\beta)\, A(\theta-\alpha)^{-1}
	=\begin{bmatrix}
	1 & \psi(\theta) \\
	0 & 1
	\end{bmatrix}$$
	with $\psi(\theta):=V(\theta-\beta)-V(\theta-\alpha)$.
	Because $V(\theta)$ is non-constant and $\alpha-\beta$ is ergodic, the function $\varphi(\theta)$ is non-identically zero. On the other hand
	$\psi(\theta)= \varphi(\theta-\alpha)$, which implies that $\psi$ is also non-identically zero. By analyticity of $V(\theta)$,
	the function $\varphi(\theta)\,\psi(\theta)$ can not be identically zero.
	In fact since  $\phi$ and $\psi$ are analytic and non-identically zero by  \L ojasiewicz inequality  the set of zeros of each of them has measure zero. Thus the same holds for $\phi \psi$, showing that it cannot be identically zero. Taking $\theta\in\T^d$ such that
	$\varphi(\theta)\,\psi(\theta)\neq 0$, the parabolic matrices
	$D_1(\theta)$ and $D_2(\theta)$ in the monodromy group
	$G_\theta(\muh)$ imply that $\Gscr_{\muh}$ is non-compact and strongly irreducible.
The conclusions then follow from Theorem~\ref{Furstenberg-Furman} and Proposition~\ref{Kifer continuity}.
\end{proof}

\begin{example}
	\label{example 4}
	Given $\mu\in\Prob(\T^d)$ and $\rho \in \Prob_c (\R)$, let
	\begin{equation}
	\label{SL2 example 4}
	\muh_E:= \mu \times \int_\R \delta_{P (\om) S_E} \, d \rho (\om)
	\end{equation}
	be the family of mixed Schr\"odinger cocycles with random frequency $\alpha$ chosen with probability $\mu$ and random noise $\omega$ chosen with probability $\rho$, corresponding to the Schr\"odinger operator~\eqref{op3}. 
\end{example}

A similar argument to the one used in Proposition~\ref{ex 1} establishes the following.

\begin{proposition}
	Consider the family of Schr\"odinger cocycles defined by the measures $\muh_E$ above.
	If $\supp (\mu)$ contains a rationally independent frequency and if 
	$\supp(\rho)$ has more than one element,  then $\muh_E$ is non-compact and strongly irreducible. Thus
	$\fle(\muh_E)>0 $ $ \forall E\in\R$ and the map $E \mapsto \fle (\muh_E)$ is continuous.
\end{proposition}

%\begin{remark}
%In all the previous examples, consider the compact set of quasi-periodic cocycles $\Sigmah$ (\ref{bigcompact}) and the measure $\muh_E$ defined there for each $E\in \R$. 
%We have seen that if $E\in \spec(H_{v,\omega,\theta})$ then $\supp(\muh_E)\subseteq \Sigmah$, which allows us to perturb the energy $E$ while ensuring the support of $\muh_E$ is always contained in a compact set $\Sigmah$. This guarantees that Proposition \ref{Kifer continuity} is still applicable and gives the continuity of the Lyapunov exponent with respect to the energy $E$. For $E$ not in the spectrum, the corresponding cocycle is uniformly hyperbolic and the continuity of LE w.r.t $E$ follows trivially.
%\end{remark}

\section{Uniform convergence results} \label{uniform}

Let $\muh\in\Prob_c(\Gscr)$ be an ergodic measure, let $\Sigmah \supset \supp (\muh)$ be a compact set, put $X = \Sigmah^\Z$ and consider the mixed cocycle driven by this measure. That is, the projections $\alfa \colon \Sigmah\to \T^d$, $\alfa(\alpha, A):=\alpha$ and
$\Ascr \colon \Sigmah\to C^0(\T^d,\R^m)$, $\Ascr(\alpha, A):=A$ define 
the dynamical system $F \colon X\times\T^d\times\R^m \to X\times \T^d\times\R^m$
$$ F(\omega,\theta, v):= \left( \sigma \omega, \theta+\alfa(\omega_0), \Ascr(\omega_0)(\theta)\, v\right) .$$

By Proposition~\ref{quasi-irred charact}, if $\muh$ is quasi-irreducible and $L_1 (\muh) > L_2 (\muh)$, then with probability one
$$\frac{1}{n} \, \log \norm{\Ascr^n (\theta) p} \to \fle (\muh)$$
for every $(\theta, \hat p) \in \T^d \times \Proj$.

The main goal of this section is to prove that taking the expectation in the random component, the above convergence is uniform. While interesting in itself, this property also provides us with the main ingredient in a future study of the statistical properties of mixed cocycles.

\subsection*{Almost invariance}
As before let $\mu:=\alfa_\ast\muh \in \Prob (\T^d)$, let $\Sigma \supset \supp (\mu)$ be a compact set and denote by $\Qop_\mu \colon \Cscr(\T^d)\to\Cscr(\T^d)$ the Markov operator
$$ (\Qop_\mu\varphi)(\theta):= \int_{\Sigma} \varphi(\theta+\alpha)\,d\mu(\alpha) .$$

\begin{definition}
\label{def almost invariance}
A  sequence $\{L_n\}_{n\geq 1}$  of functions
$\lexp_n \colon \T^d\to\R$ is said to be $\Qop_\mu$-almost invariant if it converges $m$-almost surely to some measurable function $\lexp \colon \T^d\to \R$ and there exists some constant $C<+\infty$ such that
$$ \vert \lexp_n(\theta)- (\Qop_\mu \lexp_n)(\theta)\vert \leq \frac{C}{n},\qquad \forall \; \theta\in\T^d,\; \forall\; n\geq 1. $$
\end{definition}

\begin{proposition}
\label{almost sure limit}
If a sequence of measurable functions $\lexp_n \colon \T^d\to\R$ is\,  $\Qop_\mu$-almost invariant then the limit function $L \colon \T^d\to\R$ is constant
$m$-almost everywhere.
\end{proposition}

\begin{proof}
Follows from the ergodicity of $(\Qop_\mu, m)$
and the fact that the limit function satisfies
$L(\theta)=(\Qop_\mu L)(\theta) $ for all $\theta\in\T^d$.
\end{proof}

\begin{theorem}
\label{convergence dichotomy}
Assume $(f,\mu^\Z\times m)$ is ergodic.
If a  sequence $\{\lexp_n\}_{n\geq 1}$  of continuous functions
$\lexp_n:\T^d\to\R$ is $\Qop_\mu$-almost invariant
then one of the following alternatives holds:
\begin{enumerate}
	\item[(a)]  $\lexp_n$ converges uniformly to $L$.
	\item[(b)]  There exists a dense set of points $\theta\in\T^d$  for which
	$\lexp_n(\theta)$ does not converge to $L$.
\end{enumerate}
\end{theorem}

\begin{proof}
Replacing $\lexp_n$ by $\lexp_n-L$ we can assume that the $m$-almost sure limit is $L=0$.
Suppose that $(a)$ does not hold.
Then there are subsequence $\{n_k\}$ and points $\theta^\ast_{k}\in\T^d$ such that $\vert \lexp_{n_k}(\theta^\ast_k)\vert$ does not converge to $0$.
Eventually replacing $\lexp_{n_k}$ by $-\lexp_{n_k}$ and taking a subsequence, we can assume that $\lexp_{n_k}(\theta^\ast) \to c>0$ where $\theta^\ast$ is the limit point of $\{\theta^\ast_k\}$.

\smallskip

\begin{lemma}
\label{Ln lower bound}
For all $\theta\in\T^d$ and $j\in\N$,\;
$\vert(\Qop_\mu^j \lexp_n)(\theta) -  \lexp_n(\theta)\vert \leq C\frac{j}{n}$.
\end{lemma}

\begin{proof}
For any $j\in\N$, since $\Qop_\mu$ is a Markov operator,
\begin{align*}
\vert (\Qop_\mu^j \lexp_n)(\theta) - \lexp_n(\theta) \vert
&\leq \sum_{i=0}^{j-1}
\vert (\Qop_\mu^{i+1} \lexp_n)(\theta) -
(\Qop_\mu^i \lexp_n)(\theta) \vert\\
&\leq \sum_{i=0}^{j-1}
\vert (\Qop_\mu^i)( \Qop_\mu \lexp_n  -\lexp_n )(\theta) \vert\\
&\leq \sum_{i=0}^{j-1}
\norm{\Qop_\mu \lexp_n  -\lexp_n }_\infty \leq \frac{C}{n}\,j  .
\end{align*}

\end{proof}

\begin{proposition}
\label{orbit density}
Let $\mu\in\Prob(\T^d)$ and assume that $f$ is ergodic w.r.t. $\mu^\Z\times m$.
Let $\Supp\supseteq \supp(\mu)$ be a compact set. 
Given $\varepsilon>0$ there is $\delta>0$ such that for any $n_0\in\N$ there exists   $j\geq n_0$ satisfying the following property:
for any set $\Bscr\subset \Supp^j$ with
$\mu^{\ast j}(\Bscr)<\delta$ and any $\theta\in\T^d$,
$$ \T^d = \bigcup_{  (\alpha_1,\ldots, \alpha_{j}) \in \Supp^j\setminus\Bscr  } B_\varepsilon(\theta +\alpha_1+ \cdots+\alpha_{j}) .$$
\end{proposition}

\begin{proof}
Given $\varepsilon>0$  take $p\in\N$ large enough
and continuous functions $\varphi_1,\ldots, \varphi_p\in\Cscr(\T^d)$ such that for all $i=1,\ldots, p$:
\begin{enumerate}
\item[(a)] $\sum_{j=1}^p \varphi_j = 1$,
\item[(b)] $0\leq \varphi_i\leq 1$,
\item[(c)] $\int \varphi_i\, dm=1/p$,
\item[(d)] $\diam(\supp(\varphi_i))<\varepsilon$.
\end{enumerate}
Set $\delta:=1/(4p)$ and take any $n_0\in\N$.

By~\cite[Theorem 2.3 item (8)]{CDK-paper1},
there exists $q\in \N$ such that for every $i=1,\ldots, p$,  and all $n\geq q$ we have
 $$ \frac{1}{n}\,\sum_{j=0}^{n-1}(\Qop_\mu^j \varphi_i)(\theta)\geq 1/(2p) .$$
 Hence, if $n\gg n_0$  there are many $j\geq \max\{n_0,q\}$ such that
 $$(\Qop_\mu^j \varphi_i)(\theta)\geq 1/(3 p).$$
Consider now any set  $\Bscr\subset \Supp^j$ such that
$\mu^{\ast j}(\Bscr)<\delta$. Then
\begin{align*}
\frac{1}{3p}&\leq (\Qop_\mu^j\varphi_i)(\theta)
=\int_{\Bscr}\varphi_i(\theta +  \alpha_1+\cdots+\alpha_{j})\,d\mu^{\ast j} (\alpha_1, \ldots,\alpha_{j})\\
& \quad + \int_{\Bscr\comp}\varphi_i(\theta +  \alpha_1+\cdots+\alpha_{j})\,d\mu^{\ast j} (\alpha_1, \ldots,\alpha_{j})\\
&\leq  \frac{1}{4p} + \int_{\Bscr\comp}\varphi_i(\theta +  \alpha_1+\cdots+\alpha_{j})\,d\mu^{\ast j} (\alpha_1, \ldots,\alpha_{j})\\
\end{align*}
which implies that
$$ \frac{1}{12p} \leq \int_{\Bscr\comp}\varphi_i(\theta +  \alpha_1+\cdots+\alpha_{j})\,d\mu^{\ast j} (\alpha_1, \ldots,\alpha_{j}) .$$
Thus for some $\alpha^i=(\alpha_1^i,\ldots, \alpha_{j}^i)\in \Supp^j\setminus \Bscr$ we have
$\theta +  \hat \alpha^i\in\supp(\varphi_i)$,
where $\hat \alpha^i:=\alpha_1^i+\cdots+\alpha_{j}^i$.
By (d) above,\,
$\supp(\varphi_i)\subset B_\varepsilon(\theta+\hat \alpha^i)$.
Hence
$$ \T^d = \bigcup_{i=1}^p \supp(\varphi_i)
\subset\bigcup_{i=1}^p B_\varepsilon(\theta+ \hat \alpha^i ) .$$
\end{proof}

\begin{lemma}
\label{nested boxes}
Given a closed ball $\Nscr$ with positive radius
and $n\in\N$ there exists an integer $n'>n$
and a closed ball with positive radius
$\Nscr'\subset \mathrm{int}(\Nscr)$ such that for all
$\theta\in\Nscr'$, $ L_{n'}(\theta) \geq \frac{c}{2}$.
\end{lemma}

\begin{proof}
Let $\Nscr$ be a closed  ball with radius $\varepsilon>0$
and take $\delta=\delta(\varepsilon)>0$ according to Proposition~\ref{orbit density}. Next choose $n_0\geq n$
such that for all $m=n_k\geq n_0$,
$$ (1-\frac{\delta}{8})\,c \leq   L_m(\theta^\ast)\leq
(1+\frac{\delta}{8})\,c .
$$
By Lemma~\ref{Ln lower bound}, for all $m\geq n_0$ and $j\in \N$,
\begin{equation}
\label{Q^j Lm >= c}
  (\Qop_\mu^j L_m)(\theta^\ast)  \geq (1-\frac{\delta}{8})\,c - C\,\frac{j}{m} \geq (1-\frac{\delta}{4})\,c ,
\end{equation}
where the last inequality holds provided
\begin{equation}
\label{m j threshold}
m\geq \frac{8 C j}{c \delta}\quad \Leftrightarrow \quad C\frac{j}{m}\leq \frac{\delta}{8} c.
\end{equation}

Notice  that by definition of $\Qop_\mu$ we have
$$ (\Qop_\mu^j L_m)(\theta^\ast) = \int_{\Sigma^j}
L_m(\theta^\ast+ \alpha_1 +\cdots +\alpha_{j})\, d\mu^j(\alpha_1, \ldots,\alpha_j). $$

Considering the set
$$ \Bscr_{m,j}:=\{ (\alpha_1,\ldots, \alpha_j)\in \Sigma^j\colon
  L_m(\theta^\ast+ \alpha_1 +\cdots +\alpha_{j})  \leq \frac{5}{8}\,c  \}  $$
we have
\begin{align*}
L_m(\theta^\ast+\alpha_1+\cdots+\alpha_{j}) &\leq  \frac{5 c}{8}\, \ind_{\Bscr_{m,j}} (\alpha) + \norm{L_m}_\infty \,\ind_{\Bscr_{m,j}\comp}(\alpha)     \\
&\leq  \frac{5 c}{8}\, \ind_{\Bscr_{m,j}} (\alpha) + c \, (1+\frac{\delta}{8})\,\ind_{\Bscr_{m,j}\comp}(\alpha)
\end{align*}
and integrating in $\alpha=(\alpha_1,\ldots, \alpha_j)\in \Sigma^j$
we have
$$  c \, (1-\frac{\delta}{4})\leq (\Qop_\mu^j L_m)(\theta^\ast) \leq \frac{5 c}{8}\,\mu^j(\Bscr_{m,j})+c\,(1+\frac{\delta}{8})\, (1-\mu^j(\Bscr_{m,j}))  .$$
It follows that $(3+\delta)\,\mu^j(\Bscr_{m,j}) < 3\,\delta$,
which in turn implies $\mu^j(\Bscr_{m,j}) <\delta$.

By Proposition~\ref{orbit density}, because $\Nscr$ is a ball of radius $\varepsilon$  there exists $j\geq n_0$ such that for any
  $\Bscr\subset \Sigma^j$ with $\mu^j(\Bscr)<\delta$, the set
$$\{\theta^\ast+\alpha_1+\cdots +\alpha_j \colon (\alpha_1,\ldots,\alpha_j)\in\Bscr\comp \}$$
 is $\varepsilon$-dense and in particular it has a point in
 $\mathrm{int}(\Nscr)$. Next choose $n'=m$ large enough so that inequality~\eqref{m j threshold} holds. This forces~\eqref{Q^j Lm >= c} to hold as well. Then the set $\Bscr_{m,j}$ has measure
 $\mu^j(\Bscr_{m,j})<\delta$. By the above conclusion  (from
Proposition~\ref{orbit density}) there exists
$\alpha=(\alpha_1,\ldots, \alpha_j)\in \Bscr_{m,j}\comp$ such that
$(\theta^\ast +\alpha_1+\cdots+\alpha_j)\in \mathrm{int}(\Nscr)$.
But $\alpha\notin \Bscr_{m,j}$ implies that
$L_m(\theta^\ast +\alpha_1+\cdots+\alpha_j)\geq 5 c/8$.
Thus, by continuity of $L_m$ we can find a small ball $\Nscr'$ centered at this point
and contained in $\Nscr$ such that
$L_m(\theta)\geq c/2$ for all $\theta\in \Nscr'$.
This concludes the proof.
\end{proof}

The previous lemma allows us to  construct, recursively,
a nested sequence of closed balls with positive radius
$$\Nscr_0 \supset \Nscr_1 \supset \cdots \supset \Nscr_j \supset \cdots $$
and corresponding times
$n_0 < n_1 <\cdots < n_j < \cdots $ such that
$ L_{n_j}(\theta) \geq \frac{c}{2}$ for all
$\theta\in \Nscr_j$ with $j\geq 1$. Then for any point
$\theta\in \cap_{j=1}^\infty \Nscr_j$, the sequence $\lexp_n(\theta)$ does not converge to $0$.
Since the initial ball $\Nscr_0$ is arbitrary there exists a dense set of points $\theta\in\T^d$
such that $\lexp_n(\theta)$ does not converge to $0$.  This proves that alternative $(b)$ holds.
\end{proof}

\begin{corollary}
\label{dichotomy convergence}
Let  $(f,\mu^\Z\times m)$ be ergodic. Given
a $\Qop_\mu$-almost invariant sequence of continuous functions $\lexp_n \colon \T^d\to\R$, if $\lexp_n$ converges pointwise, that is, for every $\theta\in\T^d$,  then $\lexp_n$ converges uniformly.
\end{corollary}

%\begin{proof}
%By Theorem~\ref{convergence dichotomy}, $\lexp_n$ cannot converge in a pointwise but not uniform way.
%\end{proof}

\bigskip

\subsection*{Finite scale maximal Lyapunov exponent}
%Given $\mu\in \Prob_c(\Gscr)$, let $\Sigma:= \mathrm{supp}(\mu)$.
%We are assuming that the mixed random quasi-periodic map
%$f:X\times \T^d\to X\times \T^d$ is ergodic w.r.t. $\mu^\Z\times m$.
%The measure $\mu$ determines the following linear cocycle
%$F:X\times\T^d\times\R^m \to X\times\T^d\times\R^m$, defined by
%$$ F(\omega, \theta, v) := (T\omega, \theta +\alpha(\omega_0), \Ascr(\omega)(\theta)\, v ).
%$$
Consider the sequence of continuous functions
$\dle_n \colon \T^d\times \Pp(\R^m)\to\R$ and
$ \lexp_n \colon \T^d \to\R$, respectively defined by

\begin{equation}
\dle_n(\theta, \hat p):= \frac{1}{n}\,\EE[\,\log\norm{\Ascr^n(\theta)\, p} \, ]
\end{equation}
and
\begin{equation}
\lexp_n(\theta):= \frac{1}{n}\,\EE[\,\log\norm{\Ascr^n(\theta)} \, ] \; .
\end{equation}
where the expected values stand for integrals w.r.t. $\muh^\Z$ over the space of sequences $X$.

\bigskip

\begin{proposition}
\label{pointwise convergence}
	If $\muh$ is quasi-irreducible then
	\begin{enumerate}
		\item[(1)] $\displaystyle \lim_{n\to +\infty} \dle_n(\theta, \hat p)=\fle(\muh) \quad \forall\, (\theta,\hat p)\in \T^d\times\Pp(\R^m) ,$
		\item[(2)]  $\displaystyle \lim_{n\to +\infty} \lexp_n(\theta)=\fle(\muh) \quad \forall\, \theta\in \T^d.$
	\end{enumerate}
	
\end{proposition}

\begin{proof}
Item (1) follows by applying the Dominated Convergence Theorem
to the conclusion $(1)\Rightarrow (4)$ in Proposition~\ref{quasi-irred charact}.

Item (2) simply follows from Corollary \ref{corfle}.
%we first prove $\limsup_{n\to+\infty}  \lexp_n(\theta) \leq \fle(\muh).$
%To see this, we choose the same parameters in Proposition \ref{upper LDE}. By inequality $\ref{cim}$ we have for any $\epsilon>0$, there exists $N$ such that for any $n\geq N$
%\begin{equation}
%\frac{1}{n}\log\lVert \Ascr^n(\omega,\theta)\rVert \leq \fle(\muh)+\epsilon
%\end{equation}
%holds for all $\hat\omega\notin\Bscr$ with $\muh^{\Z}(\Bscr)\leq \epsilon$. Therefore we have
%\begin{equation}
%\EE[\frac{1}{n}\log\lVert \Ascr^n(\omega,\theta)\rVert]\leq C\epsilon+\fle(\muh)+\epsilon,
%\end{equation}
%which gives $\limsup_{n\to+\infty}  \lexp_n(\theta) \leq \fle(\muh).$
%On the other hand, because $\dle_n(\theta,\hat p)\leq  \lexp_n(\theta) $,
%from  (1) we infer that $\lexp_n(\theta) $ converges pointwisely to $\fle(\muh)$,
%for every $\theta\in \T^d$.
\end{proof}

\begin{theorem}
	\label{Ln(theta) converges uniformly}
	If $\muh$ is quasi-irreducible then
	 $$ \lim_{n\to +\infty} \lexp_n(\theta)=\fle(\muh) $$
	 with uniform convergence  on $\theta\in \T^d$.
\end{theorem}

\begin{proof}
We claim that the sequence of function $\lexp_n$ is $\Qop_\mu$-almost invariant. Hence by item (2) of Proposition~\ref{pointwise convergence} and  Corollary~\ref{dichotomy convergence} we infer that
$\lexp_n(\theta)$ converges uniformly to $L_1(\muh)$.

To prove the claim we need the following inequality.
For any $g\in\SL_d(\R)$ let us write $\iota(g):=\max\{\norm{g},\norm{g^{-1}}\}$.

\begin{lemma}
\label{log nrm ggs}
Given
$g_0,g_1,\ldots, g_n\in\SL_d(\R)$,
$$ \vert \log \norm{g_n\,\cdots \, g_1}
- \log \norm{g_{n-1}\,\cdots \, g_0} \vert \leq \log\iota(g_0) +\log\iota(g_n) .$$
\end{lemma}

\begin{proof}
A simple calculation shows that for any $a,b\in\SL_d(\R)$,
$$ \vert \log \norm{a\,b} -\log \norm{a} \vert \leq \log \iota(b) , $$
and similarly
$$ \vert \log \norm{b\,a} -\log \norm{a} \vert \leq \log \iota(b) .$$
Taking $a=g_n\,g_{n-1}\,\cdots\, g_1\,g_0$ and $b=g_0^{-1}$,
from the first inequality
$$ \vert \log \norm{g_n\,\cdots \, g_1}
- \log \norm{g_{n}\,\cdots \, g_0} \vert \leq \log\iota(g_0)  .$$
Taking $a=g_n\,g_{n-1}\,\cdots\, g_1\,g_0$ and $b=g_n^{-1}$,
from the second inequality
$$ \vert \log \norm{g_{n-1}\,\cdots \, g_0}
- \log \norm{g_{n}\,\cdots \, g_0} \vert \leq \log\iota(g_n)  .$$
The conclusion of the lemma follows adding these two inequalities.
\end{proof}

Because $\muh\in\Prob_c(\Gscr)$ has compact support we have
$$\iota:= \sup_{\omega\in X}\sup_{\theta\in\T^d}
\iota( \, \Ascr(\omega)(\theta)\, )<\infty.$$
Applying Lemma~\ref{log nrm ggs} to the products of $n$ matrices, $\Ascr^n(T\omega)(\theta+\alpha(\omega_0))$
and $\Ascr^n(\omega)(\theta)$ we get
$$ \vert \log\norm{\Ascr^n(T\omega)(\theta+\alpha(\omega_0))}
-\log\norm{\Ascr^n(\omega)(\theta)} \vert \leq 2\,\iota. $$
Dividing by $n$ and integrating we get
$$ \vert (\Qop_\mu \lexp_n)(\theta)-\lexp_n(\theta) \vert
\leq \frac{2\,\iota}{n} . $$
This proves the claim.
\end{proof}

Consider the deviation sets for the mixed cocycle $F$,
\begin{align*}
\Bscr_n(\varepsilon,\theta) &:= \left\{ \omega\in X\colon  \left\vert \frac{1}{n}\,\log\norm{\Ascr^n(\omega,\theta)} - L_1(\muh)\right\vert > \varepsilon
\right\}\\
\Bscr_n^+(\varepsilon,\theta) &:= \left\{ \omega\in X\colon    \frac{1}{n}\,\log\norm{\Ascr^n(\omega,\theta)}   > L_1(\muh) + \varepsilon
\right\}\\
\Bscr_n^-(\varepsilon,\theta) &:= \left\{ \omega\in X\colon    \frac{1}{n}\,\log\norm{\Ascr^n(\omega,\theta)}  < L_1(\muh) - \varepsilon
\right\}
\end{align*}
where $\Bscr_n(\varepsilon,\theta)=\Bscr_n^+(\varepsilon,\theta)\cup \Bscr_n^-(\varepsilon,\theta)$.

\begin{theorem}
	\label{non uniform LDE}
	Assume that $\muh$ is a quasi-irreducible measure. Then for any $\varepsilon>0$ and $\delta>0$ there exists $n_0\in\N$
	such that for all $\theta\in\T^d$ and all $n\geq n_0$,\,
	$\muh^\Z( \Bscr_n(\varepsilon,\theta))<\delta$. In other words,
	$\frac{1}{n}\,\log\norm{\Ascr^n(\omega, \theta)}$ converges in measure to $\fle(\muh)$, uniformly in $\theta\in\T^d$.
\end{theorem}

\begin{proof}
	In view of~\cite[Theorem 3.1]{CDK-paper1},  we just need to prove that
	$\muh^\Z(\Bscr_n^-(\varepsilon,\theta))$ converges to $0$ uniformly in $\theta$.

	Throughout the proof we  set  $\fle:=\fle(\muh)$. We will refer to the following finite measurement
	$$ \gamma = \sup \{ \log \norm{B}_\infty \colon (\beta, B)\in \Sigmah \} . $$

	Fix $\varepsilon$ and $\delta$ positive small numbers.
	Then take
	$ 0 <\varepsilon_0 <\frac{\varepsilon\,\delta}{4-2 \delta} $ and
	$0<\delta_0< \gamma^{-1} (\varepsilon\,\delta +\varepsilon_0)$.
	These choices imply that
	\begin{equation}
	\label{delta0. varepsilon0}
	\frac{2 \varepsilon_0 + \delta_0\, \gamma}{\varepsilon+\varepsilon_0 } < \frac{\delta}{2} + \frac{\delta}{2} =\delta.
	\end{equation}
	By~\cite[Theorem 3.1]{CDK-paper1} there exists $n_0\in\N$ such that
	for all $\theta\in\T^d$ and $n\geq n_0$, \, $\muh^\Z(\Bscr_n^+(\varepsilon_0,\theta))<\delta_0$. By 	Theorem~\ref{Ln(theta) converges uniformly}  the sequence of functions $\lexp_n:\T^d\to\R$,
	$$  \lexp_n(\theta):= \EE\left[\, \frac{1}{n}\,\log\norm{\Ascr^n(\theta)} \, \right] $$
	converges uniformly to $\fle$. Hence we can assume that the threshold $n_0$
	above is chosen so that for all $\theta\in\T^d$ and $n\geq n_0$, \,
	$ \lvert\fle-\lexp_n(\theta)\rvert\leq \varepsilon_0$.

	We have that	
	\begin{align*}
	\frac{1}{n}\,\log\norm{\Ascr^n(\theta)} & \leq (\fle-\varepsilon)\,\ind_{\Bscr_n^-(\varepsilon,\theta)} +  \gamma\, \,\ind_{\Bscr_n^+(\varepsilon_0,\theta)}  \\
	&\qquad    + (\fle+\varepsilon_0)\,\ind_{ 	\Bscr_n^-(\varepsilon,\theta)\comp \cap \Bscr_n^+(\varepsilon_0,\theta)\comp} .
	\end{align*}
	Integrating these functions over $X$ we get
	$$ \lexp_n(\theta) \leq (\fle-\varepsilon) \, \chi + \gamma\, \delta_0 +
	(\fle+ \varepsilon_0) \, (1-\chi) $$
	where we write  $\chi:= \muh^\Z(\Bscr_n^-(\varepsilon,\theta))$.
	Solving in $\chi$ we have
	$$ (\varepsilon+\varepsilon_0)\,\chi \leq \gamma\,\delta_0 + \fle-\lexp_n(\theta)+\varepsilon_0 \leq \gamma\,\delta_0 + 2\,\varepsilon_0 $$
	which by~\eqref{delta0. varepsilon0} implies that
	$$  \muh^\Z(\Bscr_n^-(\varepsilon,\theta)) = \chi \leq \frac{\gamma\,\delta_0 + 2\,\varepsilon_0}{\varepsilon+\varepsilon_0} <\delta . $$
	Since $\theta\in\T^d$ and $n\geq n_0$ are arbitrary, the proof is finished.
\end{proof}

\subsection*{Directional Lyapunov exponents} We are now ready to state and prove the main result of this section.

\begin{theorem}
	If $\muh$ is quasi-irreducible and $\fle(\muh)>L_2(\muh)$ then
	$$ \lim_{n\to +\infty} \dle_n(\theta, \hat p)=L_1(\muh) $$
	 with uniform convergence  on $(\theta,\hat p)\in \T^d\times\Pp(\R^m)$.
\end{theorem}

\begin{proof}
Assume by contradiction that the convergence is not uniform, then there exists a sequence of $\{(\theta_n, \hat p_n)\}$ such that $\dle_n(\theta_n, \hat p_n)$ is strictly bounded (by a constant $C>0$) away from $L_1(\muh)$ for all $n$ (we should bear in mind that here $\{n\}$ is a subsequence). Passing to a subsequence, we have $\theta_n \to \theta \in \T^d$.

Given a matrix $g\in \GL_m(\R)$, denote by $\medp(g)\in\Pp(\R^m)$  the most expanding direction of $g$ and by $\medr(g)$ any unit vector in the class $\medp(g)$. Let $\medp_n(\omega,\theta_n):=\medp\left( \Ascr^n(\omega)(\theta_n)\right)$ be the most expanding direction of the matrix $\Ascr^n(\omega)(\theta_n)$.
Similarly, for any $\theta\in \T^d$, let $\medp_\infty(\omega,\theta)$ denote the $\muh^\Z$-almost sure limit of the point $\medp_n(\omega,\theta)$ (this is ensured by a simple adaptation of~\cite[Proposition 4.4]{DKLEbook} which we shall explain in detail below). Unit vector representatives of these projective points will be respectively denoted by $\medr_n(\omega,\theta)$ and $\medr_\infty(\omega,\theta)$.

\begin{lemma}
	\label{medir conv probability}
For every $\theta\in\T^d$ and every sequence $\theta_n\to \theta$ in $\T^d$, up to choosing a subsequence of $\{\theta_n\}$, the following convergence in probability holds
$$ \lim_{n\to +\infty} \muh^\Z \{\omega\in X\colon \,
d\left( \medp_n(\omega,\theta_n),
\medp_\infty(\omega,\theta)\right) > c\,\} =0\qquad \forall\, c>0 . $$
\end{lemma}

\begin{proof}

Fix a small number $\delta>0$ and $0<\varepsilon <\lambda:=\frac{1}{12}\,( \fle(\muh)-L_2(\muh))$.
%The sequence of functions (of the variable $\theta$)	
%\begin{align*}
% \frac{1}{n}\,\EE\left[ \, \log  \frac{\norm{\Ascr^{2n}(\theta)}}{\norm{\Ascr^{n}(\theta+\tau^n)}\,\norm{\Ascr^{n}(\theta)}} \, \right] &= 2\, L_{2n}(\theta)-\lexp_n(\theta)-(\Qop_\mu^n \lexp_n)(\theta)
%\end{align*}
%converges uniformly to zero. Notice that by the Markov property of $\Qop$
%$$ \norm{\Qop_\mu^n \lexp_n - \fle(\muh)}_\infty
%= \norm{\Qop_\mu^n (\lexp_n - \fle(\muh))}_\infty
%\leq \norm{ \lexp_n - L_1(\muh) }_\infty .$$
%Hence, if we denote
%$$h_n(\omega, \theta):=\frac{1}{n}\, \log  \frac{\norm{\Ascr^{2n}(\omega)(\theta)}}{\norm{\Ascr^{n}(T^n \omega)(\theta+\tau^n(\omega))}\,\norm{\Ascr^{n}(\omega)(\theta)}}\leq 0,$$
%then for any fixed $\theta\in \T^d$,
%$$ \Pp[\, h_n(\omega, \theta) \leq -\varepsilon \,] \leq -\frac{1}{\varepsilon}\, \EE[\, h_n(\omega, \theta) \,]\,  \text{ converges uniformly  to }\, 0. $$
%There exists $n_0=n_0(\varepsilon,\delta)\in\N$ such that for all $n\geq n_0$ and $\theta\in\T^d$, for $\omega$ outside the set $\Bscr_n=\Bscr_n(\theta):= [\, h_n(\omega, \theta) \leq -\varepsilon \,]$ with $\Pp(\Bscr_n)<\delta/4$ we have
%\begin{equation}
%\label{angle cond}
%\frac{\norm{\Ascr^{2n}(\omega)(\theta)}}{\norm{\Ascr^{n}(T^n \omega)(\theta+\tau^n(\omega))}\,\norm{\Ascr^{n}(\omega)(\theta)}}> e^{-n\,\varepsilon} .
%\end{equation}

The gap ratio of a matrix  $g\in \GL_m(\R)$ is defined by
$$\gpr(g):=\frac{s_2(g)}{s_1(g)}:= \frac{\norm{\wedge_2 g}}{\norm{g}^2} .$$

%By the definition of $\lambda$, Theorem~\ref{fiber upper ldt} and Corollary \ref{corfle}, for any $\theta\in \T^d$, for $\muh^\Z$-almost every $\omega\in X$,
%$$
%\limsup_{n\to \infty} \frac{1}{n}\, \log  \frac{\norm{\wedge_2 \Ascr^{n}(\omega)(\theta)}}{\norm{\Ascr^{n}(\omega)(\theta)}^2} =  -2\,\lambda <0 .$$

By~\cite[Theorem 3.1]{CDK-paper1} and Theorem~\ref{non uniform LDE}, there exists $n_1(\delta,\epsilon,\lambda)$ such that for any $n>n_1$, there exists a Borel set
$\Bscr_n'(\theta,\lambda):=\Bscr_n'(\theta)\subset X$ with $\muh^\Z(\Bscr_n'(\theta))<\delta/2$
such that for all $\theta\in\T^d$ and $\omega\notin \Bscr_n'(\theta)$
\begin{equation}\label{limsu}
\limsup_{n\to \infty} \frac{1}{n}\, \log  \norm{\wedge_2 \Ascr^{n}(\omega)(\theta)}\leq 2L_1(\muh)-12\lambda.
\end{equation}
and
\begin{equation}\label{lim}
\left\vert \frac{1}{n}\,\log\norm{\Ascr^n(\omega,\theta)} - L_1(\muh)\right\vert \leq \varepsilon.
\end{equation}
Therefore, we have
$$ \frac{1}{n}\, \log  \frac{\norm{\wedge_2 \Ascr^{n}(\omega)(\theta)}}{\norm{\Ascr^{n}(\omega)(\theta)}^2} <-10\lambda .$$
This implies that
\begin{equation}
\label{gap cond}
\frac{ s_2(\Ascr^{n}(\omega)(\theta))}{\norm{\Ascr^{n}(\omega)(\theta)}} = \frac{\norm{\wedge_2 \Ascr^{n}(\omega)(\theta)}}{\norm{\Ascr^{n}(\omega)(\theta)}^2} \leq e^{-10n\,\lambda} .
\end{equation}

Take any $n\geq n_1$ and consider now the exceptional set	
$$ \Bscr_n^\ast(\theta):=\Bscr_n'(\theta_n)\cup \Bscr_n'(\theta)$$	
which has measure $\muh^\Z(\Bscr_n^\ast(\theta))<\delta$ uniformly in $\theta$. Therefore, we have that the final exceptional set $\Omega^c:=\bigcup_{\theta\in \T^d} \Bscr_n^\ast(\theta) \times \{\theta\}$ satisfies $(\muh^\Z\times m)(\Omega^c)<\delta$.

By Lemma 4.2 in~\cite{DKLEbook}, for $\delta>0$ given above, there exist $r\in \N$ such that for any $N,n\in \N$ with $n\geq r$ and $N\geq rn$, we can construct a $\delta$-doubling sequence $\{m_i\}_{i\geq0}$ such that $m_0=n$, $m_k=N$ for some $k\geq 1$ and $\lvert m_i-2m_{i-1}\rvert<\delta m_i$ for all integers $i\geq 1$.

By Lemma 4.3 in~\cite{DKLEbook}, for $\delta>0$ given above and the measurable set $\Omega$ with $(\muh^\Z\times m)(\Omega)>1-\delta$, there exists a measurable subset $\Omega_0\subset \Omega$ with $(\muh^\Z\times m)(\Omega_0)>1-C\delta$ ($C$ is numerical constant) and there are integers $n\geq r$ such that for all $(\omega, \theta)\in \Omega_0$ and $N\geq rn$, we have the $\delta$-doubling sequence $\{m_i\}_{i\geq 0}$ such that $m_0=n$, $m_k=N$ for some $k\geq 1$ and more importantly $f^{m_i}(\omega,\theta)\in \Omega$ for all $0\leq i<k$.

It is not difficult to see that the proof of~\cite[Lemma 4.3]{DKLEbook} can be modified to ensure that for any $\theta\in \T^d$, the set of $\omega\in X$ such that $(\omega,\theta)\in \Omega_0$ is of measure $1-C\delta$; for that, the stopping time $m (x)$ in this proof should be chosen uniformly in $\theta\in\T^d$.
This is possible essentially because by~\cite[Lemma 2.5]{CDK-paper1}, the convergence in Birkhoff's ergodic theorem for continuous observables is uniform in the variable $\theta$. In order to use this result, one should approximate the indicator function $\ind_\Omega$ by a continuous function and take into account the special form of the set $\Omega$.
We leave the remaining details to the reader.

%Moreover, by~\cite[Lemma 2.5]{CDK-paper1} we have the uniform convergence in every $\theta\in \T^d$ of the Birkhoff sums for $\muh^\Z$-almost every $\omega\in X$. Note also that our $\muh^\Z\times m$ is regular. Therefore take two compact sets $A,B$ with $A\subset B\subset \Omega$, $(\muh^\Z\times m)(A)\geq 1-3\delta$ and $(\muh^\Z\times m)(B)\geq 1-2\delta$, by approximating the indicator function $\mathds{1}_{\Omega}$ with continuous functions $\varphi$ being constant $1$ in $A$, supported on $B$ and vanishes outside of $B$, we can adapt Lemma \ref{Birkhoff LDE1} to the proof of Lemma 4.3 in~\cite{DKLEbook} so that $\Omega_0$ satisfies that for any $\theta\in \T^d$, the set of $\omega\in X$ such that $(\omega,\theta)\in \Omega_0$ is of measure $1-C\delta$ (this is ensured by the special form of $\Omega$ along with the fact that ``$m(x)$'' in the proof is again well defined for $\muh^\Z$-a.e. $\omega$ and is independent of $\theta\in\T^d$).

Given $(\omega,\theta) \in \Omega_0$, up to enlarging $n$ by $2n$, the pairs of matrices $\Ascr^{m_i}(\omega,\theta), \Ascr^{m_{i+1}-m_i}(f^{m_i}(\omega,\theta))$; $\Ascr^{m_i}(\omega,\theta_n), \Ascr^{m_{i+1}-m_i}(f^{m_i}(\omega,\theta_n))$ satisfy the gap condition (\ref{gap cond}) with $\kappa_{ap}:= e^{-\,10m_i\lambda}$, $0\leq i<k$.
On the other hand, again by $(\ref{limsu})$ and $(\ref{lim})$ we have
$$
\frac{\norm{\Ascr^{m_{i+1}}(\omega,\theta)}}{\norm{\Ascr^{m_{i+1}-m_i}(f^{m_i} (\omega,\theta))}\,\norm{\Ascr^{m_i}(\omega,\theta)}}> e^{-2m_{i+1}\,\varepsilon}>e^{-5m_i\,\varepsilon}
$$
for all $\theta\in \T^d$ and $0\leq i<k$. Thus the angle assumption is also fulfilled with $\varepsilon_{ap}:= e^{-5\,m_i\varepsilon}$. Notice that by the choice of $\varepsilon$
we have $\kappa_{ap}\ll \varepsilon_{ap}^2$.

This ensures that we can apply the Avalanche Principle for every pair $(m_i,m_{i+1})$, $0\leq i <k$.

By the Avalanche Principle, given $c>0$, we have
for any $(\omega,\theta)\in \Omega_0$, for any $0\leq i<k$
$$ d(\medp_{m_i}(\omega,\theta), \medp_{m_{i+1}}(\omega,\theta) )<\frac{c}{4^{i+1}} $$
and
$$ d(\medp_{m_i}(\omega,\theta_n), \medp_{m_{i+1}}(\omega,\theta_n) )<\frac{c}{4^{i+1}} .$$

This immediately shows that
\begin{equation}\label{aslimit}
d(\medp_{n}(\omega,\theta), \medp_{N}(\omega,\theta) )<\frac{c}{3} \end{equation}
and
$$ d(\medp_{n}(\omega,\theta_n), \medp_{N}(\omega,\theta_n) )<\frac{c}{3}.$$

Note that since $\delta>0$ can be taken arbitrarily small, by $(\ref{aslimit})$ we have that for any $\theta\in \T^d$, $\medp_n(\omega,\theta)$ converges $\muh^\Z$-almost surely to $\medp_\infty(\omega,\theta)$.

On the other hand, since $\theta_n$ converges to $\theta$,
the sequence of matrices $\Ascr^n(\omega)(\theta_n)$ converges to
	   $\Ascr^n(\omega)(\theta)$ uniformly in $(\omega,\theta)\in X$. Passing to a subsequence of $\{\theta_n\}$, we obtain
$$ d( \medp_n(\omega,\theta_n) ), \medp_n(\omega,\theta)  )<\frac{c}{3} .$$
Adding up the previous inequalities we have,
for $(\omega,\theta) \in \Omega_0$,
\begin{align*}
d(\medp_N(\omega,\theta_n), \medp_N(\omega,\theta) ) &\leq
d(\medp_N(\omega,\theta_n), \medp_n(\omega,\theta_n))  \\
&\quad + \,
d( \medp_n(\omega,\theta_n), \medp_n(\omega,\theta)  ) \\
&\quad + d( \medp_n(\omega,\theta), \medp_N(\omega, \theta) )   )\\
&<\frac{c}{3}  + \frac{c}{3} + \frac{c}{3} = c.
\end{align*}

Therefore,
$$ \lim_{n\to +\infty} \muh^\Z \{\omega\in X\colon \,
d\left( \medp_n(\omega,\theta_n),
\medp_n(\omega,\theta)\right) > c\,\} =0\qquad \forall\, c>0 .$$

As noted above, $\medp_n(\omega,\theta)$ converges $\muh^\Z$-almost surely to $\medp_\infty(\omega,\theta)$ for any $\theta\in \T^d$,
and hence it also converges in measure. This completes the proof.
\end{proof}

Note that the hyperplane $\medr_\infty(\omega,\theta)^\perp$ is the sum of all
invariant subspaces in Oseledets decomposition associated with Lyapunov exponents $<\fle(\mu)$.  By Lemma~\ref{medir conv probability},
\begin{align}
\label{conv measure}
\frac{\norm{\Ascr^n(\omega)(\theta_n)\,p_n}}{\norm{\Ascr^n(\omega)(\theta_n)}}
 \geq \vert \medr_n(\omega, \theta_n)\cdot p_n\vert \longrightarrow
\vert \medr_\infty(\omega,\theta)\cdot p \vert
\end{align}
with convergence in probability.
Hence, if $\medr_\infty(\omega,\theta)\cdot p=0$
then
$$\limsup_{n\to +\infty} \frac{1}{n}\,\log\norm{\Ascr^n(\omega)(\theta) p}<\fle(\mu).$$
Therefore, because
$\frac{1}{n}\,\log\norm{\Ascr^n(\omega)(\theta) p}$ converges almost surely
to $\fle(\mu)$, we must have $ \medr_\infty(\omega,\theta)\cdot p \neq 0$
for $\muh^\Z$-almost every $\omega$.

The convergence in measure in~\eqref{conv measure}
implies that there exists a subsequence $n_j$ of integers such that
$\vert \medr_{n_j}(\omega,\theta_{n_j})\cdot p_{n_j}\vert $ converges to $\vert \medr_\infty(\omega,\theta)\cdot p \vert  $  for $\muh^\Z$-almost every $\omega\in X$.
Thus,  possibly passing to a subsequence again, we can assume that $\muh^\Z$-almost surely
\begin{align*}
\liminf_{n\to \infty} \frac{\norm{\Ascr^n(\omega)(\theta_n)\,p_n}}{\norm{\Ascr^n(\omega)(\theta_n)}}
>0  .
\end{align*}
This implies that
$\frac{1}{n}\,\log \frac{\norm{\Ascr^n(\omega)(\theta_n)\,p_n}}{\norm{\Ascr^n(\omega)(\theta_n)}}$ converges to zero $\muh^\Z$-almost surely.
Therefore, by the Dominated Convergence Theorem, and because
$\lexp_n(\theta)$ converges uniformly to $\fle(\muh)$,
\begin{align*}
&\lim_{n\to +\infty} \dle_n(\theta_n, \hat p_n) \\
=&\lim_{n\to +\infty} \frac{1}{n}\, \EE\left[\, \log \norm{\Ascr^n(\theta_n)\, p_n} \,\right]\\
=& \lim_{n\to +\infty} \frac{1}{n}\, \EE\left[\, \log \norm{\Ascr^n(\theta_n)} \,\right] + \lim_{n\to +\infty} \frac{1}{n}\, \EE\left[\, \log \frac{\norm{\Ascr^n(\theta_n)\,p_n}}{\norm{\Ascr^n(\theta_n)}} \,\right]\\
=&\lim_{n\to +\infty}   \lexp_n(\theta_n) + 0 = \fle(\muh) .
\end{align*}
This contradicts to $\dle_n(\theta_n, \hat p_n)$ being bounded away from $L_1(\muh)$ by a constant $C>0$.
This concludes the proof.
\end{proof}

\subsection*{Acknowledgments} The second author would like to thank Alex Furman for indicating him the correct references for his work on the positivity criterion.

The first author was supported  by Funda\c{c}\~{a}o para a Ci\^{e}ncia e a Tecnologia, under the project PTDC/MAT-PUR/29126/2017.
The second author was supported  by Funda\c{c}\~{a}o para a Ci\^{e}ncia e a Tecnologia, under the projects: UID/MAT/04561/2013 and   PTDC/MAT-PUR/29126/2017.
The third author was supported by the CNPq research grants 306369/2017-6 and 313777/2020-9 and by the Coordena\c{c}\~ao de Aperfei\c{c}oamento de Pessoal de N\'ivel Superior - Brasil (CAPES) - Finance Code 001.
\bigskip

\providecommand{\bysame}{\leavevmode\hbox to3em{\hrulefill}\thinspace}
\providecommand{\MR}{\relax\ifhmode\unskip\space\fi MR }
\providecommand{\MRhref}[2]{%
  \href{http://www.ams.org/mathscinet-getitem?mr=#1}{#2}
}
\providecommand{\href}[2]{#2}

\end{document}